\documentclass[reqno,a4paper]{amsart}
\usepackage{enumerate}
\usepackage{amsfonts,amssymb,amsmath,amsthm}
\usepackage{epsfig}
\usepackage{graphics}
\usepackage[normalem]{ulem}
\usepackage{color}
\usepackage{version}
\usepackage{caption}
\usepackage{subcaption}
\usepackage{hyperref}
\usepackage{tikz}

\definecolor{gG}{RGB}{ 60, 186,  84}
\definecolor{gY}{RGB}{244, 194,  13}
\definecolor{gB}{RGB}{48., 88.6667, 158.}
\definecolor{gR}{RGB}{219,  50,  54}

\tikzstyle{mycircle}=[circle,draw=black!80,thick, minimum size=1em]

\newtheorem{thm}{Theorem}[section]
\newtheorem{theorem}[thm]{Theorem}
\newtheorem{main theorem}[thm]{Main Theorem}
\newtheorem{corollary}[thm]{Corollary}
\newtheorem*{main}{Main Theorem}
\newtheorem{lemma}[thm]{Lemma}

\newtheorem{prop}[thm]{Proposition}
\newtheorem{conjecture}[thm]{Conjecture}

\theoremstyle{definition}

\newtheorem{defn}[thm]{Definition}

\newtheorem{remark}[thm]{Remark}

\newcommand{\size}[1]{\mathrm{size}(#1)}
\newcommand{\ceil}[1]{\lceil #1 \rceil}

\newcounter{relctr} 
\everydisplay\expandafter{\the\everydisplay\setcounter{relctr}{0}} 

\newcommand\labelrel[2]{%
  \begingroup
    \refstepcounter{relctr}%
    \stackrel{\textnormal{(\arabic{relctr})}}{\mathstrut{#1}}%
    \originallabel{#2}%
  \endgroup
}
\AtBeginDocument{\let\originallabel\label} 

\DeclareMathOperator{\Int}{int}
\DeclareMathOperator{\tr}{tr}
\newcommand{\Chat}[0]{\widehat{\mathbb{C}}}
\newcommand{\Arc}[2]{\textrm{Arc}_#1\left(#2\right)}
\newcommand{\Sec}[2]{\textrm{Sec}_#1\left(#2\right)}

\def\CQ{\ensuremath{\mathbb{Q}[i]}}

\title[Complexity of the independence polynomial]{Zeros, chaotic ratios and the computational complexity of approximating the independence polynomial }
\author{David de Boer}
\author{Pjotr Buys}
\author{Lorenzo Guerini}
\author{Han Peters}
\author{Guus Regts}
\date{\today}
\thanks{$\ddagger$ DdB and PB are funded by the Netherlands Organisation of Scientific Research (NWO): 613.001.851}
\address[David de Boer, Pjotr Buys, Lorenzo Guerini, Han Peters, Guus Regts]{Korteweg de Vries Institute for Mathematics, University of Amsterdam. P.O. Box 94248  
1090 GE Amsterdam  
The Netherlands}
\email{\{daviddeboer2795,pjotr.buys,lorenzo.guerini92,hanpeters77,guusregts\}@\texttt{gmail.com}}
\begin{document}

\begin{abstract}
The independence polynomial originates in statistical physics as the partition function of the hard-core model.
The location of the complex zeros of the polynomial is related to phase transitions, and plays an important role in the design of efficient algorithms to approximately compute  evaluations of the polynomial.

In this paper we directly relate the location of the complex zeros of the independence polynomial to computational hardness of approximating evaluations of the independence polynomial. 
We do this by moreover relating the location of zeros to chaotic behaviour of a naturally associated family of rational functions; the occupation ratios.

\end{abstract}
\maketitle

\section{Introduction}

The independence polynomial of a graph $G = (V,E)$ is defined by 
$$
Z_G(\lambda) = \sum_{I \subseteq V} \lambda^{|I|},
$$
where the sum is taken over all \emph{independent} subsets $I$ of the vertex set $V$. Recall that $I$ is said to be independent if no two vertices in $I$ are connected by an edge. Note that $Z_G(1)$ equals the number of independent subsets of $V$.

In statistical physics the independence polynomial is known as the partition function of the hard-core model. Of particular interest from a physics perspective is the location of the zeros of the partition function for certain classes of graphs. Away from these zeros the free energy is analytic, i.e. there are no phase transitions in the Lee-Yang sense cf. ~\cite{LYI52}.

It turns out that exact computation of the independence polynomial for large graphs is not feasible for most values of $\lambda$; it is a \#P-Hard problem\footnote{The complexity class \#P may be seen as the counting version of the complexity class NP. For example, the problem of deciding whether a graph on $n$ vertices contains an independent set of size $k$ is a problem in NP, while the problem of determining the number of independent sets of size $k$ is in \#P. See~\cite{Val79,AB09} for further background.}, cf.~\cite{R96,G00,V01}. 
A question that has received significant interest is for which $\lambda \in \mathbb C$ there exist polynomial time algorithms that approximate $Z_G(\lambda)$, up to small multiplicative constants. See e.g.~\cite{Weitz,Slysun,Barbook,PaR17,Galanisetal20,Anarietal} and the references therein.

Surprisingly, much like absence of zeros implies absence of phase transitions (in the Lee-Yang sense), absence of zeros implies the existence of efficient algorithms for this computational problem. 
More formally, on the maximal simply connected open set containing the origin on which the independence polynomial does not vanish for all graphs of a given maximum degree $\Delta$ there exists an efficient algorithm for approximating the independence polynomial~\cite{Barbook,PaR17}. Let us denote this maximal `zero-free' set by $\mathcal{U}_\Delta$.
For real values of $\lambda$ in the complement of the closure of $\mathcal{U}_\Delta$, approximating the partition function is computationally hard~\cite{Scott_2005,Slysun,PR19,Galanisetal20}. 
In other words, the absence/presence of complex zeros near the real axis marks a transition in the computational complexity of approximating the independence polynomial  of graphs of bounded degree $\Delta$ for real values of $\lambda$. The transition point for positive $\lambda$ coincides with the phase transition for the hard-core model on the Cayley tree of degree $\Delta$.

\begin{figure}
    \centering
    \includegraphics{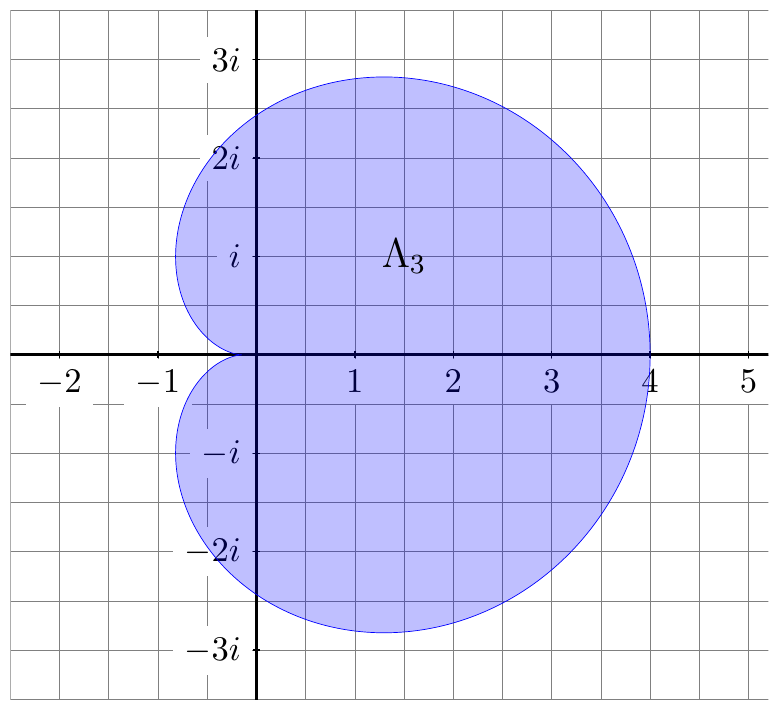}
    \caption{The cardioid $\Lambda_3$.}
    \label{fig:cardioid3}
\end{figure}

A natural question is whether a similar phenomenon manifests itself for non-real $\lambda$.
Bezakov\'a, Galanis, Goldberg and \v{S}tefankovi\v{c}~\cite{Galanisetal20} made an important contribution towards solving this question, by showing that for any integer $\Delta\geq 3$ and non-real $\lambda$ outside a certain cardioid\footnote{Although the domain $\Lambda_\Delta$ resembles a cardioid, it is formally not a cardioid. However, as discussed in section \ref{sec: Cayley trees}, it plays an analogous role as the Main Cardioid of the Mandelbrot set, justifying our use of the term cardioid.}, $\Lambda_\Delta$, approximation of the independence polynomial for graphs of bounded degree at most $\Delta$ is computationally hard. (In fact \#P-hard.) See Figure \ref{fig:cardioid3} for a picture of $\Lambda_3$ and Definition \ref{def:cardioiddef} for the definition of $\Lambda_\Delta$. 
Earlier it was shown by two of the authors of the present paper~\cite{PR19} that zeros of the independence polynomial of graphs of maximum degree at most $\Delta$ accumulate on the entire boundary of  $\Lambda_\Delta$.
In particular the `zero-free' set $\mathcal{U}_\Delta$ is contained in the cardioid; their intersections with the real axis in fact coincide~\cite{Scott_2005,PR19}. 
Buys~\cite{buys20} however showed that $\Lambda_\Delta$ does contain zeros of the independence polynomial of graphs of bounded degree $\Delta$. This in particular indicates that the result of~\cite{Galanisetal20} does not fully answer the question how zeros relate to computational hardness for non-real $\lambda$. 

The goal of the present paper is to solve this question by directly relating, for any fixed integer $\Delta\geq 3$, the zeros for the family of graphs of maximum degree at most $\Delta$ to the parameters where approximating evaluations of the independence polynomial is computationally hard. 
Our result is obtained by studying a natural family of rational maps associated to this family of graphs, using techniques and ideas from complex dynamics. We show that `chaotic behaviour' of this family is equivalent to the presence of zeros, and implies computational hardness.

\subsection{Occupation ratios}
Given $\Delta \in \mathbb{Z}_{\geq 2}$, we define $\mathcal{G}_\Delta$
as the collection of finite simple rooted graphs $(G,v)$ such that the maximum degree of $G$ is at most $\Delta$.
For $i \in \{1, \dots, \Delta\}$ we define
$
    \mathcal{G}_\Delta^i = \{(G,v) \in \mathcal{G}_\Delta: \deg(v) \leq i\}.
$
The \emph{occupation ratio}, or \emph{ratio} for short, of a rooted graph $(G,v)$ is defined by the rational function
$$
R_{G,v}(\lambda) := \frac{Z_G^\mathrm{in}(\lambda)}{Z_G^\mathrm{out}(\lambda)},
$$
where ``in'' means that in the definition of $Z_G(\lambda)$ the sum is taken only over independent sets $I$ that contain the marked point $v$, while ``out'' means that the independent sets do not contain $v$. 
The ratio is a very useful tool in studying the zeros of the independence polynomial, see Lemma \ref{lemma:equivalence} below, and has been key in several of the aforementioned works. 
The ratio is also relevant from a statistical physics perspective as it is closely related to the \emph{free energy}.


When $(G,v)$ is a rooted Cayley tree of depth $n-1$ and down-degree $d = \Delta-1$, the ratio satisfies
$$
R_{G,v}(\lambda) = f_{\lambda,d}^{n}(0),
$$
where 
$$
f_{\lambda,d}(z) := \frac{\lambda}{(1+z)^d}
$$
and throughout the paper we write $f^n$ for the $n$-th iterate of the map $f$.

In this context it is therefore natural to consider $\lambda$ as the parameter which determines the orbit of the marked point $0$. This type of setting is often studied in complex dynamical systems, where one is interested in the sets where the parameter $\lambda$ is \emph{active} or \emph{passive}. A parameter $\lambda_0$ is said to be passive if the family of rational functions $\{\lambda \mapsto f_{\lambda,d}^{\circ n}(0)\}$ is normal at $\lambda_0$, i.e. there exists a neighborhood such that every sequence in this family has a subsequence that converges uniformly. A parameter is active if it is not passive. The most well-known activity-locus is undoubtedly the boundary of the Mandelbrot set, where the iterates of the functions $z^2+c$ are considered.
Following this terminology we define the \emph{activity-locus}, $\mathcal{A}_\Delta$, by
\[
\mathcal{A}_\Delta:=\{\lambda_0\in \mathbb{C} \mid  \{\lambda\mapsto R_{G,v}(\lambda)\mid (G,v)\in \mathcal{G}_\Delta\} \text{ is not locally normal at } 
\lambda_0 \}. 
\]

Another notion of chaotic behaviour of the ratios appears in the proof of the result of Bezakov\'a, Galanis, Goldberg and \v{S}efankovi\v{c}~\cite{Galanisetal20}.
An important step towards proving \#P-hardness is showing that for every non-real $\lambda$ outside of the closed cardioid $\overline{\Lambda_\Delta}$ the set of values $\{R_{G,v}(\lambda)\mid (G,v)\in \mathcal{G}_{\Delta}^1\}$ is dense in $\hat{\mathbb C}$. 
Motivated by this we define
\[
\mathcal{D}_\Delta:=\{\lambda\in \mathbb{C} \mid \{R_{G,v}(\lambda)\mid (G,v)\in \mathcal{G}_{\Delta}^1\} \text{ is dense in }\hat{\mathbb C}\}
\]
and refer to the closure of $\mathcal{D}_\Delta$ as the \emph{density-locus}.
We will prove it is equal to the activity-locus, thereby showing that these two notions of chaotic behaviour of the ratios are essentially equivalent.

\subsection{Main result}
To state our main result connecting the presence of zeros to computational hardness, 
we define the \emph{zero-locus} as the closure of
$$
\mathcal{Z}_\Delta = \{\lambda \in \mathbb{C}: Z_G(\lambda) = 0 \text{ for some } G \in \mathcal{G}_\Delta\}.
$$
We informally define the \emph{$\#\mathcal{P}$-locus} as the closure of the collection of $\lambda$ for which approximating $Z_G(\lambda)$ is \#P-hard for $G\in \mathcal{G}_\Delta$. See subsection \ref{subsec: Plocus} below for a formal definition.

The main results of this paper can now be stated succinctly as follows.
\begin{main}
    \label{thm: main theorem}
    For any integer $\Delta \geq 3$ the zero-locus, the activity-locus and the density-locus are equal and contained in the $\#\mathcal{P}$-locus. In other words:
    $$
        \overline{\mathcal{Z}_\Delta} = \mathcal{A}_\Delta = \overline{\mathcal{D}_{\Delta}} \subseteq \overline{\#\mathcal{P}_{\Delta}}.
    $$
\end{main}

We remark that the topological structure of the complement of the zero-locus is not yet understood.
We have the following conjecture.
\begin{conjecture}\label{conj:connected}
For each integer $\Delta\geq 3$, the set $\mathbb{C}\setminus \overline{\mathcal{Z}_\Delta}$ is connected.
\end{conjecture}
Should this conjecture be true, then by Proposition~\ref{prop:simply connected} below, we know that the maximal `zero-free' set containing $0$, $\mathcal{U}_\Delta$, equals the complement of the zero-locus.
Since there exists a polynomial time algorithm~\cite{Barbook,PaR17} for approximating the independence polynomial on $\mathcal{U}_\Delta$, this would imply with our main theorem a complete understanding of the computational complexity of approximating the independence polynomial in terms of the location of the zeros as well as in terms of chaotic behaviour of the ratios.

\begin{remark}
We note that~\cite{PR20} and~\cite{buys2020lee} combined implicitly contain similar equivalent characterizations for the Lee-Yang zeros of the partition function of the ferromagnetic Ising model on bounded degree graphs. In that setting the complement of the zero-locus is in fact connected when the edge interaction parameter is sub-critical.
\end{remark}





\subsection{Computational complexity}\label{subsec: Plocus}
We formally state here the computational problems we are interested in.
We denote by $\CQ$ the collection of complex numbers with rational real and imaginary part. 
Let $\lambda\in \CQ$, $\Delta\in \mathbb{N}$ and consider the following computational problems.
\begin{itemize}
\item[\emph{Name}] \#Hard-CoreNorm($\lambda,\Delta)$
\item[\emph{Input}] A graph $G$ of maximum degree at most $\Delta$.
\item[\emph{Output}] If $Z_G(\lambda)\neq0$ the algorithm must output a rational number $N$ such that $N/1.001\leq |Z_G(\lambda)|\leq 1.001N$. If $Z_G(\lambda)=0$ the algorithm may output any rational number.
\end{itemize}

\begin{itemize}
\item[\emph{Name}] \#Hard-CoreArg($\lambda,\Delta)$
\item[\emph{Input}] A graph $G$ of maximum degree at most $\Delta$.
\item[\emph{Output}] If $Z_G(\lambda)\neq0$ the algorithm must output a rational number $A$ such that $ |A-a|\leq\pi/3$ for some $a\in \arg(Z_G(\lambda)$. If $Z_G(\lambda)=0$ the algorithm may output any rational number.
\end{itemize}

We can now formally define the $\#\mathcal{P}$-locus, as the closure of the set, 
\[
\#\mathcal{P}_{\Delta} := \{\lambda \in \CQ: \text{ the problem \#Hard-CoreNorm($\lambda,\Delta)$ is \#P-hard}\}.
\]
We remark that in the definition of $\#\mathcal{P}_\Delta$ we could also replace \#Hard-CoreNorm$(\lambda,\Delta)$ by \#Hard-CoreArg$(\lambda,\Delta)$ without altering the validity of Theorem~\ref{thm: main theorem}.

We moreover note that the constant $1.001$ is rather arbitrary. It originates from~\cite{Galanisetal20}. As remarked there the constant can be replaced by any other constant.
Let us quickly explain the idea. If say \#Hard-CoreNorm$(\lambda,\Delta)$ is \#P-hard, but there would be a polynomial time algorithm for the problem with $1.001$ replaced by $1.001^2$, then we could run this algorithm on the disjoint union of two copies of the same graph $G$ obtaining an a $1.001^2$ approximation to the norm of $Z_{G\cup G}(\lambda)=Z_{G}(\lambda)^2$. This would immediately gives us a $1.001$-approximation to the norm of $Z_G(\lambda)$. Since the number of vertices of $G\cup G$ is polynomial in the number of vertices of $G$, we would thus also get a polynomial time algorithm for the problem with constant $1.001$.

\medskip

\noindent {\bf Organization.} 
After introducing preliminary definitions and results in section 2, we treat the degree $\Delta = 2$ case in section 3. While the equalities between different loci are different when $\Delta = 2$, the explicit descriptions of the zero- and activity-locus will be used in the higher degree cases.

In section 4 we prove the equality of the zero-locus and the activity-locus. The inclusion of the latter in the former is actually an immediate consequence of Montel's Theorem, and proved earlier in Corollary~\ref{cor:zerofree->normal}.
We end that section by showing that connected components of the complement of the zero-locus are simply connected.

In section 5 we prove the equality of the activity- and the density-locus, and in section 6 we prove that the density-locus is contained in the $\#\mathcal{P}$-locus. 

We end our paper by discussing a special subclass of graphs: the finite Cayley trees of fixed down-degree $\Delta-1$. In this setting classical results from complex dynamical systems can be used to obtain a precise description of the zero- and activity-locus. While there zeros do not lie in the activity-locus, the activity-locus equals the accumulation set of the zeros.

\section{Preliminaries}

In this section we collect some preliminary results and conventions that will be used in the remainder of the paper.
The results in this section are not necessarily new, but often cannot be found in the literature in the exact way they are stated here. For convenience of the reader we include proofs, especially when the methods are similar to those used later in the paper.

\subsection{Ratios of graphs and trees.}

Recall that for a rooted graph $(G,v)$ the occupation ratio is defined as the following rational function in $\lambda$
$$
R_{G,v}(\lambda) = \frac{Z_G^\mathrm{in}(\lambda)}{Z_G^\mathrm{out}(\lambda)}.
$$
We note that $Z_G(\lambda) = Z_G^\mathrm{in}(\lambda) + Z_G^\mathrm{out}(\lambda)$, which implies that $Z_G(\lambda) = 0$ if and only if $R_{G,v}(\lambda) = -1$, unless $Z_{G,v}^\mathrm{in}(\lambda)$ and $Z_{G,v}^\mathrm{out}(\lambda)$ both vanish, in which case the value of the rational function $R_{G,v}(\lambda)$ may not equal $-1$. The next lemma will show that we can often ignore this difficulty.

We will write $G-v$ for the graph $G$ with vertex $v$ removed, and $G-N[v]$ for the graph with $N[v]$ removed, where $N[v] = \{u \in V(G): \{u,v\} \in E(G)\} \cup \{v\}$ is the closed neighborhood of $v$. We observe that
$Z_{G,v}^{out} (\lambda) = Z_{G-v}(\lambda)$, and similarly $Z_{G,v}^{in} (\lambda) = \lambda \cdot Z_{G-N[v]}(\lambda)$.

\begin{lemma}\label{lemma:equivalence}
Let $\lambda \in \mathbb{C}^*$. The following three statements are equivalent.
\begin{enumerate}
    \item There exists a graph $G$ of maximum degree at most $\Delta$ for which $Z_G(\lambda) = 0$.
    \item There exists a rooted graph $(G,v) \in \mathcal{G}_\Delta$ for which $R_{G,v}(\lambda) = -1$.
    \item There exists a rooted graph $(G,v) \in \mathcal{G}_\Delta$ for which $R_{G,v}(\lambda) \in \{-1,0,\infty\}$.
\end{enumerate}
\end{lemma}
\begin{proof}
Assume that (1) holds, then there is a graph $G$ of maximum degree at most $\Delta$ for which $Z_G(\lambda) = 0$. Without loss of generality we can assume $G \in \mathcal{G}_\Delta$ satisfies $Z_G(\lambda) = 0$ and has a minimal number of vertices, i.e. for any graph $H \in \mathcal{G}_\Delta$ with $Z_H(\lambda)=0$ we have $|V(G)| \leq |V(H)|$.
For any vertex $v \in V(G)$ we have 
\[
0 = Z_G(\lambda) = Z_{G,v}^{in}(\lambda) + Z_{G,v}^{out}(\lambda).
\]
As $|V(G-v)| < V(G)$ we have $Z_{G,v}^{out}(\lambda)= Z_{G - v}(\lambda) \neq 0$, which implies $R_{G,v}(\lambda) = -1$.
Thus (2) holds. Trivially, if (2) holds then also (3) holds. To complete the proof we will assume (3) holds and show that (1) follows.

Assume there is a rooted graph $(G,v) \in \mathcal{G}_\Delta$ for which $R_{G,v}(\lambda) \in \{-1,0,\infty\}$. 
If $R_{G,v} (\lambda) = -1$, we either have $Z_{G,v}^{out} (\lambda) = 0$, in which case (1) follows, or $Z_{G,v}^{out} (\lambda) \neq 0$, in which case $Z_{G,v}^{in}(\lambda) = -Z_{G,v}^{out}(\lambda)$ and (1) follows as well.
If $R_{G,v} (\lambda) = \infty$ we have $Z_{G, v}^{out} (\lambda) = 0$. As $Z_{G, v}^{out} (\lambda) = Z_{G - v}(\lambda)$ we see (1) holds. The final case is $R_{G,v} (\lambda) = 0$, in which case we have $0 = Z_{G,v}^{in}(\lambda) = \lambda \cdot Z_{G - N[v]}(\lambda)$. 
Now as $\lambda \neq 0$, we must have $Z_{G - N[v]}(\lambda) = 0$, which concludes the proof.
\end{proof}

Note that for $\lambda = 0$ we have $R_{G,v} (\lambda) = 0$ and $Z_G(\lambda) = 1$ for any graph $G$ and any vertex $v \in V(G)$. Hence for $\lambda = 0$, statements (1) and (2) in Lemma~\ref{lemma:equivalence} are still equivalent, while statement (3) is not equivalent to (1) or (2).

The following result due to Bencs~\cite{Bencs18} will play an important role in this paper.

\begin{theorem}
    \label{thm: stable paths}
    Let $(G,v)\in \mathcal{G}_\Delta^i$ be a rooted connected graph. Then there is a rooted tree $(T,u)\in \mathcal{G}_\Delta^i$ and induced graphs $G_1,\ldots,G_k$ of $G$ such that
    \begin{itemize}
        \item [(i)] $Z_T=Z_G\prod_{i=1}^k Z_{G_i}$,
        \item[(ii)] $R_{G,v} = R_{T,u}$.
    \end{itemize}
\end{theorem}

The following result 
 is a consequence.

\begin{lemma}
    \label{lem: zero implies -1 at leaf}
    Let $\lambda \in \mathbb{C}$ and $(G,v) \in \mathcal{G}_\Delta$ with $Z_G(\lambda) = 0$.
    Then there is a rooted tree $(T,u) \in \mathcal{G}_\Delta ^1$ such that $Z_T(\lambda)=0$ and $R_{T,u}(\lambda)=-1$.
\end{lemma}
\begin{proof}
Note that for any graph $G$ we have $Z_G(0) = 1$, so we can assume $\lambda \neq 0$. By Lemma~\ref{lemma:equivalence} there exists a rooted graph $(G,v) \in \mathcal{G}_\Delta$ such that $R_{G,v}(\lambda) = -1$.
By Theorem~\ref{thm: stable paths}(i) we see there is a rooted tree $(T,u) \in \mathcal{G}_\Delta$ with $Z_T(\lambda)=0$.
It follows there is a tree $\tilde{T}$ of maximum degree $\Delta$ with a minimal number of vertices such that $Z_{\tilde{T}}(\lambda) = 0$. 
For $\tilde{T}$ and any vertex $v \in V(\tilde{T})$ we have $R_{\tilde{T}, v}(\lambda) = -1$. The lemma follows by choosing $v$ a leaf of $\tilde{T}$.
\end{proof}

At a later stage we will need to worry about the degree of the root vertex in our definition of the activity- and density-locus.
We therefore introduce some definitions to facilitate their discussion.

Fix an integer $\Delta\geq 2$ throughout.
For $i=1,\ldots,\Delta$ we denote the family of ratios with root degree at most $i$ by
\[
\mathcal{R}^i_\Delta:=\{R_{G,v)}\mid (G,v)\in \mathcal{G}^i_\Delta\}.
\]
We just write $\mathcal{R}_\Delta$ instead of $\mathcal{R}^\Delta_\Delta$.
For a given $\lambda\in \mathbb{C}$, we denote the set of values of these ratios by
\[
\mathcal{R}^i_\Delta(\lambda):=\{R_{G,v)}(\lambda)\mid (G,v)\in \mathcal{G}^i_\Delta\}.
\]
Then we define $\mathcal{A}^i_\Delta$ to be the collection of $\lambda_0$ at which the family $\mathcal{R}^i_\Delta$ is not normal.
We just write $\mathcal{A}_\Delta$ instead of $\mathcal{A}^\Delta_\Delta$.
Finally, we introduce $\mathcal{D}_\Delta^i$ to be the collection of $\lambda$ for which the set $\mathcal{R}^i_\Delta(\lambda)$ is dense in $\mathbb{C}$.
Note that we denote $\mathcal{D}^1_\Delta$ by $\mathcal{D}_\Delta$ (as opposed to the above convention).

\subsection{Graph manipulations and definition of the cardioid}

The recursion formula given in the following lemma is well known.

\begin{lemma}
    \label{lem: tree recursion}
Let $T = (V,E)$ be a tree and $v$ a vertex of $T$. Suppose $v$ is connected to $d\geq 1$ other vertices $u_1, \dots, u_d$. Denote $T_s$ for the tree that is the connected component of $T - v$ containing $u_s$. Then we have
\begin{equation}\label{eq:treeformula1independence}
    R_{T, v}(\lambda) =  \frac{\lambda}{\prod_{s=1}^d( 1 + R_{T_s, u_s}(\lambda))}.
\end{equation}
\end{lemma}
\begin{proof}
We have
\begin{equation}\label{eq:stepHJKLFGHJKL}
 \frac{Z_{T, v}^{in}(\lambda) }{Z_{T, v}^{out}(\lambda) } =  \lambda \frac{Z_{T- N[v]}(\lambda) }{Z_{T - v}(\lambda) } = \lambda \prod_{s=1}^d \frac{Z_{T_s,u_s}^{out}(\lambda) }{Z_{T_s}(\lambda) } = \lambda \prod_{s=1}^d \frac{Z_{T_s,u_s}^{out}(\lambda) }{Z_{T_s,u_s}^{out}(\lambda)  + Z_{T_s,u_s}^{in}(\lambda) },
\end{equation}
where in the second equality we use that the partition function of a graph factors into the partition functions of its connected components.
By dividing for each $s \in \{1, \ldots, d\}$ the denominator and enumerator of the right hand side of equation \eqref{eq:stepHJKLFGHJKL} by $Z_{T_s, u_s}^{out}(\lambda) $ we obtain the desired formula.
\end{proof}

This lemma implies the claim from the introduction that the ratios of Cayley trees are given by iterating $f_{\lambda,d}(z) = \frac{\lambda}{(1+z)^d}$. We refer to Section \ref{sec: Cayley trees} for an in-depth discussion of Cayley trees and their associated dynamics.

\begin{defn}\label{def:cardioiddef}
Define the cardioid $\Lambda_\Delta$ as the closure of the set of parameters $\lambda$ for which $f_{\lambda,d}$ has an attracting fixed point.
\end{defn}
 
Note that $0 \in \Lambda_\Delta$. One can show, see Section 2.1 in ~\cite{PR19}, that
\[
\Lambda_\Delta =\bigg \{\frac{z}{(1-z)^\Delta}\ |\  |z| \leq \frac{1}{\Delta-1} \bigg \}.
\]
Taking $z=\tfrac{-1}{\Delta-1}$, we observe that 
\[
\lambda^*(\Delta)=\frac{-(\Delta-1)^{\Delta-1}}{\Delta^\Delta}
\]
is the intersection point of $\Lambda_\Delta$ with the negative real line.
 
Let $G=(V,E)$ be a graph and let $(G_i, v_i)$ be rooted graphs for $i\in V$. 
We refer to the graph obtained from $G$ and the $G_i$ by identifying each vertex $i\in V$ with $v_i$ as \emph{implementing} the $G_i$ in $G$, see Figure \ref{graph}.
The next lemmas describe the effect on the ratios for various choices of $G$ and $G_i.$

\begin{lemma}\label{lemma: paths}
Let $P_n$ denote the path on $n$ vertices. Let $(G_i, v_i)$ be rooted graphs for $i \in \{1, \ldots, n\}$ and denote $\mu_i(\lambda) = R_{G_i, v_i}(\lambda)$. Let $\tilde{P}_n$ be the graph obtained by implementing the $G_i$ in $P_n$.
Then 
\[
    R_{\tilde{P}_n, v_n}(\lambda)  =  (f_{\mu_n(\lambda) } \circ \cdots \circ f_{\mu_1(\lambda) })(0),
\]
where $f_{\mu}(z) = \frac{\mu}{1+z}$.
\end{lemma}
\begin{proof}
We use induction on $n$. For $n=1$, by definition we have $R_{G_1, v_1}(\lambda) =\mu_1(\lambda)  =f_{\mu_1(\lambda) }(0)$. As $\tilde{P}_1 = G_1$, we have $R_{\tilde{P}_1, v_1}(\lambda)  = R_{G_1, v_1}(\lambda)$. The base case follows.

Suppose the statement holds for some $n \geq 1$. The vertex $v_{n+1}$ has 1 neighbor that is part of the path $P_n$. Let us denote that neighbor as $v_n$. It follows that
\begin{align*}
    R_{\tilde{P}_{n+1},v_{n+1}}(\lambda)&= \frac{Z^{in}_{\tilde{P}_{n+1},v_{n+1}}(\lambda)}{Z^{out}_{\tilde{P}_{n+1},v_{n+1}}(\lambda)}=\frac{Z^{in}_{G_{n+1},v_{n+1}}(\lambda)}{Z^{out}_{G_{n+1},v_{n+1}}(\lambda)}\cdot \frac{Z^{out}_{\tilde{P}_n, v_n}(\lambda)}{Z_{\tilde{P}_n, v_n}(\lambda)}\\
 &=R_{G_{n+1},v_{n+1}}(\lambda)\cdot \frac{Z^{out}_{\tilde{P}_n, v_n}(\lambda)}{Z^{out}_{\tilde{P}_n, v_n}(\lambda)+ Z^{in}_{\tilde{P}_n, v_n}(\lambda)}= \frac{R_{G_{n+1},v_{n+1}}(\lambda)}{1+R_{\tilde{P}_n, v_n}(\lambda)}\\
 &= f_{\mu_{n+1}(\lambda)}(R_{\tilde{P}_n, v_n}(\lambda)),
\end{align*}
where in the second equality we use that the partition function of a graph factors into the partition functions of its connected components.

By the induction hypothesis we have 
$$
R_{\tilde{P}_n, v_n}(\lambda) = (f_{\mu_n(\lambda)} \circ \cdots \circ f_{\mu_1(\lambda)})(0),
$$
from which it follows that 
$$
R_{\tilde{P}_{n+1}, v_{n+1}}(\lambda) = (f_{\mu_{n+1}(\lambda)} \circ f_{\mu_n(\lambda)} \circ \cdots \circ f_{\mu_1(\lambda)})(0),
$$
completing the proof.
\end{proof}

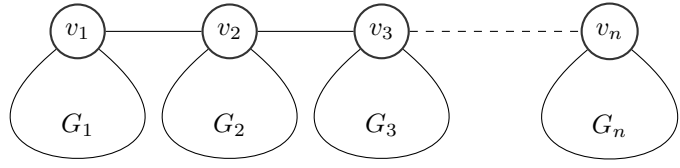
\begin{figure}[ht]
    \centering
\begin{tikzpicture}[label distance=6mm, every loop/.style={min distance=30mm,in=220,out=320}]
\node[mycircle,label=below:{$G_1$}](a) {$v_1$};
\node[mycircle,right of=a,node distance=2cm,label=below:{$G_2$}](b) {$v_2$};
\node[mycircle,right of=b,node distance=2cm,label=below:{$G_3$}](c) {$v_3$};
\node[mycircle,right of=c,node distance=3cm,label=below:{$G_n$}](d) {$v_n$};

\draw[-] (a) to node {{}} (b);

\draw[-] (b) to node {{}} (c);

\draw[dashed] (c) to node {{}} (d);

\path[every node/.style={font=\sffamily\small}]
    (a) edge [loop below] node {} (a)
    (b) edge [loop below] node {} (b)
    (c) edge [loop below] node {} (c)
    (d) edge [loop below] node {} (d);
   
\end{tikzpicture}
    \caption{The graph $\tilde{P}_n$ in Lemma~\ref{lemma: paths}}
    \label{graph}
\end{figure}

\begin{remark}
Note that if the graphs $G_i$ in Lemma~\ref{lemma: paths} are all of maximum degree $\Delta$ and the roots $v_i$ have degree at most $\Delta-2$ for $ i \in \{ 2, \ldots, n-1\}$ and at most degree $\Delta-1$ for $i \in \{1,n\}$, then the graph $\tilde{P}_n$ is also of maximum degree $\Delta$. 
\end{remark}

\begin{lemma}\label{cor:implementing}
Let $G = (V, E)$ be a graph and denote $n = |V|$. Let $(H,v)$ be a rooted graph. Let $\tilde{G} = (\tilde{V}, \tilde{E})$ be obtained from $G$ by implementing $n$ copies of $(H,v)$ in $G$. Then for any $w \in V$ we have 
\begin{equation}\label{eq:implementingpartitionfunctions}
    \frac{Z_{\tilde{G}, w} (\lambda)}{(Z_{H,v}^{out} (\lambda))^n} = Z_{G, w} (R_{H,u}(\lambda))
\end{equation}
and
\begin{equation}\label{eq:implementratios}
    R_{\tilde{G}, w} (\lambda) = R_{G, w} (R_{H,v}(\lambda)).
\end{equation}
\end{lemma}
\begin{proof}
We have
\begin{equation}\label{eq:helpin}
 \frac{Z_{\tilde{G}, w}^{in} (\lambda)}{(Z_{H,v}^{out} (\lambda))^n} =  \frac{\sum_{\substack{ I \in \mathcal{I}(G) \\w \in I }} Z_{H,v}^{in}(\lambda)^{|I|} Z_{H,v}^{out}(\lambda)^{n-|I|}  }{(Z_{H,v}^{out} (\lambda))^n} = Z_{G, w}^{in} (R_{H,v}(\lambda))
\end{equation}
and
\begin{equation}\label{eq:helpout}
 \frac{Z_{\tilde{G}, w}^{out} (\lambda)}{(Z_{H,v}^{out} (\lambda))^n} =  \frac{\sum_{\substack{ I \in \mathcal{I}(G) \\w \not \in I }} Z_{H,v}^{in}(\lambda)^{|I|} Z_{H,v}^{out}(\lambda)^{n-|I|}  }{(Z_{H,v}^{out} (\lambda))^n} = Z_{G, w}^{out} (R_{H,v}(\lambda)).
\end{equation}
Equality (\ref{eq:implementingpartitionfunctions}) follows from equalities \eqref{eq:helpin} and \eqref{eq:helpout} noting that for any graph $W$ and any vertex $u$ of $W$ we have $Z_{W} (\lambda) = Z_{W,u}^{in} (\lambda) + Z_{W,u}^{out} (\lambda)$. Equality (\ref{eq:implementratios}) follows from equalities \eqref{eq:helpin} and \eqref{eq:helpout} and the definition of the ratio.
\end{proof}

We will also need the following slight variation on Lemma \ref{lem: tree recursion}.

\begin{lemma}\label{lem:mergingtrees}
Let $(G_1,v_1)$ and $(G_2, v_2)$ be rooted graphs, and define the rooted graph $(\tilde{G},\tilde{v})$ by identifying the roots $v_1$ and $v_2$.
Then 
\[
R_{\tilde{G},\tilde{v}}(\lambda) = \lambda^{-1} \cdot R_{G_1,v_1}(\lambda)\cdot R_{G_2,v_2}(\lambda).
\]
\end{lemma}
\begin{proof}
We compute
\[
R_{\tilde{G},\tilde{v}}(\lambda) = \frac{Z_{\tilde{G},\tilde{v}}^{in}(\lambda)}{Z_{\tilde{G},\tilde{v}}^{out}(\lambda)} = \frac{Z_{G_1,v_1}^{in}(\lambda)\cdot Z_{G_2,v_2}^{in}(\lambda) \cdot \lambda ^{-1}}{Z_{G_1,v_1}^{out}(\lambda)\cdot Z_{G_2,v_2}^{out}(\lambda)} = \lambda^{-1} \cdot R_{G_1,v_1}(\lambda)\cdot R_{G_2,v_2}(\lambda).
\]
\end{proof}

\subsection{The Shearer region}
Denote the open disk around $0$ with radius $\frac{(\Delta-1)^{\Delta-1}}{\Delta^\Delta}$ by $B_\Delta$. 
This region, also known as the Shearer region, is the maximal open disk centered around $0$ that is zero free for the independence polynomial of graphs of maximum degree $\Delta$~\cite{Scott_2005,Shearer1985}.
We will show the Shearer region is also disjoint from the activity-locus and the density-locus, which will later be used to deal with the $\lambda =0$ case in the proof of our main theorem.

\begin{lemma}\label{lem: RegionAroundlambda=0}
Let $\Delta \geq 2$ be an integer. 
Then $B_\Delta$ is disjoint from the activity-locus, the zero-locus and the density-locus, i.e., we have
 $B_\Delta \cap  \overline{\mathcal{D}_\Delta}= B_\Delta \cap  \mathcal{A}_\Delta= B_\Delta \cap  \overline{\mathcal{Z}_\Delta}= \emptyset$.
\end{lemma}
\begin{proof}
We claim that for any rooted graph $(G,v) \in \mathcal{G}_\Delta$ and any $\lambda \in B_\Delta$ we have
\[
|R_{G,v}(\lambda)| <
\begin{cases}
 \frac{1}{\Delta} \ \ \ \ \ \text{ if } \deg(v) \leq \Delta - 1,\\
\frac{1}{\Delta-1} \ \ \text{ otherwise.}
\end{cases}
\]
By Theorem~\ref{thm: stable paths} we can equivalently work with rooted trees $(T,v) \in \mathcal{G}_\Delta$ instead of rooted graphs.

We will proof the claim by induction on the number of vertices of $T$. If $|V(T)| = 1$, we have $\deg(v) =0$ and therefore $R_{T,v}(\lambda) = \lambda$. The claim then follows as $\frac{(\Delta-1)^{(\Delta-1)}}{\Delta^{\Delta}} < \frac{1}{\Delta}$ for all $\Delta \geq 2$. Suppose the claim holds for all rooted trees $(T,v) \in \mathcal{G}_\Delta$ with $|V(T)| \leq n$ for some $n\geq 1$. Let $(\tilde{T},\tilde{v}) \in \mathcal{G}_\Delta$ be a rooted tree with $n+1$ vertices. Denote the $d$ children of $\tilde{v}$ as $u_1, \ldots, u_d$ and denote $(T_i, u_i)$ for the rooted subtree of $\tilde{T}$ with root $u_i$. By Lemma~\ref{lem: tree recursion} we have
\[
R_{\tilde{T},\tilde{v}}(\lambda) = \frac{\lambda}{\prod_{i=1}^d (1+R_{T_i, u_i}(\lambda))}.
\]
We note that each $(T_i, u_i)$ has at most $n$ vertices, hence the induction hypotheses applies. Furthermore in $T_i$ we have $\deg(u_i) \leq \Delta -1$ as $\tilde{T}$ has maximum degree at most $\Delta$. Thus we see
\begin{align*}
    |R_{\tilde{T},\tilde{v}}(\lambda)| &= \frac{|\lambda|}{\prod_{i=1}^d |1+R_{T_i, u_i}(\lambda)|} \leq \frac{|\lambda|}{\prod_{i=1}^d (1-|R_{T_i, u_i}(\lambda)|)} \\
    &< \frac{|\lambda|}{(1-\frac{1}{\Delta})^d} = \frac{\Delta^d |\lambda|}{(1-\Delta)^d} < \frac{(\Delta-1)^{\Delta-1-d}}{\Delta^{\Delta-d}}.
\end{align*}
Now if $d\leq \Delta - 1$, we see $\frac{(\Delta-1)^{\Delta-1-d}}{\Delta^{\Delta-d}} < \frac{1}{\Delta}$ hence the claim follows for that case. If $d=\Delta$ we have $\frac{(\Delta-1)^{\Delta-1-d}}{\Delta^{\Delta-d}} = \frac{1}{\Delta-1}$,
which proves the claim. 

It follows from the claim above that the family of ratios $\mathcal{R}_\Delta$ maps $B_\Delta$ into the open unit disk, for all $\Delta \geq 2$. 
So clearly $B_\Delta \cap \mathcal{D}_\Delta = \emptyset$. As $B_\Delta$ is open, we have $B_\Delta \cap \overline{\mathcal{D}_\Delta} = \emptyset$.

By Montel's Theorem the family  $\mathcal{R}_\Delta$ is normal on $B_\Delta$, so $B_\Delta \cap \mathcal{A}_\Delta = \emptyset$. 
We showed for all rooted graphs $(G,v) \in \mathcal{G}_\Delta$ that $|R_{G,v} (\lambda)| < \frac{1}{\Delta-1} \leq 1 $, hence the ratio will never equal $-1$. For $\lambda \neq 0$, we see by 
Lemma~\ref{lemma:equivalence} that $\lambda \not \in \mathcal{Z}_\Delta$. For $\lambda = 0$ we note that $Z_G(0) = 1$ for any graph $G \in \mathcal{G}_\Delta$. 
It follows that $B_\Delta \cap \mathcal{Z}_\Delta = \emptyset$. 
Again, as $B_\Delta$ is open, we have $B_\Delta \cap  \overline{\mathcal{Z}_\Delta}= \emptyset$. This completes the proof.
\end{proof}



\begin{remark}\label{rmk:shearerandcardioid}
We note that on the negative real line the Shearer region agrees with the interior of the cardioid, i.e we have $\mathbb{R}_{\leq 0} \cap B_\Delta = \mathbb{R}_{\leq 0} \cap \Int{\Lambda_\Delta}$ for all integers $\Delta \geq 3$.
\end{remark}

Lemmas~\ref{lem: RegionAroundlambda=0} and \ref{lemma:equivalence} together imply one of the inclusions in our main result.

\begin{corollary}
    \label{cor:zerofree->normal}
    For all $\Delta \geq 2$ the activity-locus is contained in the zero-locus, i.e. $\mathcal{A}_\Delta \subseteq \overline{\mathcal{Z}_\Delta}$.
\end{corollary}
\begin{proof}
Equivalently, we want to show that for any $ \lambda \in \mathbb{C} \setminus \overline{\mathcal{Z}_\Delta}$ the family $\mathcal{R}_\Delta$ is normal at $\lambda$. 
By Lemma~\ref{lem: RegionAroundlambda=0} this is the case for $\lambda = 0$ and thus we
assume that $\lambda \neq 0$. Take a sufficiently small neighborhood $U$ 
around $\lambda$ such that $0 \not \in U$ and
$U \cap \overline{\mathcal{Z}_\Delta} = \emptyset$. It follows from 
Lemma~\ref{lemma:equivalence} that the family $\mathcal{R}_\Delta$ 
avoids $\{-1, 0, \infty\}$ for all $\lambda' \in U$. Hence by Montel's Theorem the 
family is normal on $U$.
\end{proof}

\section{Graphs with maximum degree at most two}\label{sec: Degree 2}
In this section we will deal with graphs of maximum degree at most two, in other words graphs for which each component is a path or a cycle. 
We will show that
\[
    \overline{\mathcal{Z}_{2}} = \mathcal{A}_2 = (-\infty,-1/4]
    \quad
    \text{ and }
    \quad
    \mathcal{D}_{2}^1 = \mathcal{D}_{2}^2= \emptyset.
\]
An explicit description of $\mathcal{Z}_{2}$ was already known~\cite{HL72,Scott_2005}; we provide a new proof for the sake of completeness.

Note that this is in contrast to the situation for $\Delta \geq 3$ as stated in Theorem~\ref{thm: main theorem}.
It follows from Lemma~\ref{lem: zero implies -1 at leaf} that $\mathcal{Z}_2$
is equal to the set of $\lambda$ for which there is a $(T,v) \in \mathcal{G}_2^1$,
with $T$ a tree, such that $R_{T,v}(\lambda) = -1$. The collection $\mathcal{G}_2^1$
consists of rooted graphs where the component containing the root is a path rooted at an endpoint. 
Let $(P_n,v_n)$ denote a path
on $n$ vertices rooted at an endpoint $v_n$. If we let $f_{\lambda}(z) = \lambda/(1+z)$
then it follows from Lemma~\ref{lemma: paths} that
$R_{P_n,v_n}(\lambda) = f_{\lambda}^n(0)$. For fixed $\lambda$ the map 
$f_\lambda$ is a M\"obius transformation and therefore we first review some properties of M\"obius transformation.

\subsection{M\"obius transformations}
\label{sec: Mobius transformations}
Everything that is done in this section can for example be found in
\cite[Section 4.3]{Beardon1995}.
Let $\mathcal{M}$ denote the group of M\"obius transformations with composition
as group operation and let $\textrm{GL}_2(\mathbb{C})$ denote the group 
of $2\times2$ invertible matrices with complex entries. The following map is a surjective 
group homomorphism.
\[
	\Phi:\textrm{GL}_2(\mathbb{C}) \to \mathcal{M},\quad 
	\Big(
	\begin{array}{cc}
	 a & b \\
	 c & d \\
	\end{array}
	\Big) \mapsto \Big(z \mapsto \frac{a z + b}{c z + d}\Big).
\]
For any $g \in \mathcal{M}$ take an element $A \in \Phi^{-1}(\{g\})$
and define $\tr^2(g) = \tr(A)^2/\det(A)$. This quantity does not depend
on the choice of $A$ and thus $\tr^2$ is a well defined function on $\mathcal{M}$.
We say that elements $f,g \in \mathcal{M}$ are conjugate if there exists 
an $h \in \mathcal{M}$ such that $f = h \circ g \circ h^{-1}$. 
\begin{lemma}[{\cite[Theorem 4.3.4]{Beardon1995}}]
    \label{lem: Mobius classification}
    Let $f,g \in \mathcal{M}$ not equal to the identity. 
    The maps $f,g$ are conjugate if and only if $\tr^2(f) = \tr^2(g)$.
    It follows that $g$ is conjugate to
    \begin{itemize}
    \item
    a rotation $z \mapsto e^{i \theta} \cdot z$ for some $\theta \in (0,\pi]$
    if and only if $\tr^2(g) \in [0,4)$;
    \item
    the translation $z \mapsto z+1$ if and only if $\tr^2(g) = 4$;
    \item 
    a multiplication $z \mapsto \xi \cdot z$ for some $\xi \in \mathbb{C}^*$
    with $|\xi| < 1$ if and only if $\tr^2(g) \in \mathbb{C}\setminus [0,4]$.
    \end{itemize}
    The map $g$ is said to be elliptic, parabolic or loxodromic in these 
    three cases respectively.
\end{lemma}

Observe that if $f = h \circ g \circ h^{-1}$ then $f^n = h \circ g^n \circ h^{-1}$.
It follows that, to understand the dynamical behaviour of a M\"obius transformation
$g$, it is enough to understand the dynamical behaviour of any element in 
the conjugacy class of $g$. If $g$ is loxodromic then it has two distinct fixed points
in $\Chat$, one of which is attracting and the other is repelling. Under iteration of 
$g$ the orbit of every initial value except for the repelling fixed point converges to the
attracting fixed point. If $g$ is parabolic then $g$ has a unique fixed point, and
under iteration of $g$ all orbits converge to this fixed point. If
$g$ is elliptic then $g$ is conjugate to a rotation $z \mapsto e^{i \theta} \cdot z$.
We say that $g$ is conjugate to a \emph{rational rotation} if $\theta$ is a rational multiple
of $\pi$ and otherwise we say that $g$ is conjugate to an \emph{irrational rotation}. If $g$ 
is conjugate to a rational rotation there is a positive integer $n$ such that
$g^n$ is equal to the identity. If $g$ is conjugate to an irrational rotation 
it has two fixed points, say $p,q$, and $\Chat\setminus\{p,q\}$ is foliated by generalized 
circles on which $g$ acts conjugately to an irrational rotation.

We end this subsection by classifying $f_\lambda$ in terms of its parameter.
\begin{lemma}
    \label{lem: f_lambda classification}
    The M\"obius transformation $f_\lambda$ is 
    \begin{itemize}
        \item 
        elliptic if $\lambda \in (-\infty,-1/4)$;
        \item 
        parabolic if $\lambda = -1/4$;
        \item
        loxodromic if $\lambda \in \mathbb{C}^*\setminus (-\infty,-1/4]$.
    \end{itemize}
\end{lemma}
\begin{proof}
This follows from Lemma~\ref{lem: Mobius classification} and the fact
that $\tr^2(f_\lambda) = -1/\lambda$.
\end{proof}

\subsection{Determining the zero and activity-locus}
In this subsection we we will show that both $\overline{\mathcal{Z}_2}$
and $\mathcal{A}_2^1$ are equal to $(-\infty, -1/4]$. By definition 
we have $\mathcal{A}_2^1 \subseteq \mathcal{A}_2$ and by 
Corollary~\ref{cor:zerofree->normal} we have
$\mathcal{A}_2 \subseteq \overline{\mathcal{Z}_2}$, hence it will follow that $\mathcal{A}_2$
is equal to $(-\infty, -1/4]$ as well.
\begin{lemma}\label{lem:zerosDelta2}
    Zeros of $Z_G$ for graphs $G \in \mathcal{G}_2$ form a dense subset of the interval $(-\infty, -1/4)$, hence $\overline{\mathcal{Z}_2} = (-\infty, -1/4]$.
\end{lemma}

\begin{proof}
We claim that $\lambda \in \mathcal{Z}_2$ if and only if $f_\lambda$ is conjugate
to a rational rotation. 

First suppose that $\lambda \in \mathcal{Z}_2$. Then, 
by Lemma~\ref{lem: zero implies -1 at leaf}, there is an $n \geq 1$ such 
that for the path on $n$ vertices $P_n$ rooted at the endpoint $v_n$ we have 
$R_{P_n,v_n}(\lambda) = -1$ and thus $f_\lambda^n(0) = -1$. Because 
$f_\lambda^2(-1) =0$ regardless of the value of $\lambda$ we obtain that
$f_\lambda^{n+2}(0) = 0$. This means that $0$ is a periodic point of $f_\lambda$
of period strictly larger than $1$. This can only occur if $f_\lambda$ is conjugate
to a rational rotation, as is explained in Section~\ref{sec: Mobius transformations}.

Suppose that $f_\lambda$ is conjugate to a rational rotation.
Note that this implies that $\lambda$ is not equal to zero. 
Take the smallest positive integer $n$ such that $f_\lambda^n$ is equal
to the identity and thus specifically $f_\lambda^n(0) = 0$. Note that
$f_\lambda(0) = \lambda$ and $f_{\lambda}^2(0) = \lambda/(1+\lambda)$ 
and thus $n \geq 3$. Since $f_\lambda^{-2}(0) = -1$ 
we obtain that $R_{P_{n-2,v_{n-2}}}(\lambda)=f_\lambda^{n-2}(0) = -1$. 
It follows from the proof of Lemma~\ref{lemma:equivalence} that $\lambda$ is a root of $Z_{P_{n-2}}$.

Parameters $\lambda$ for which $f_\lambda$ is conjugate to a rational rotation
lie dense in the set of parameters for which $f_\lambda$ is conjugate to any 
rotation. It follows from Lemma~\ref{lem: f_lambda classification} that
$\overline{\mathcal{Z}_2} = (-\infty, -1/4]$.
\end{proof}

We remark that $\tr^2$ of the map that sends $z$ to $e^{i\theta}\cdot z$
is equal to $2(1+\cos(\theta))$. By comparing this to the value of 
$\tr^2(f_\lambda)$ it follows from the previous proof that 
\[
    \mathcal{Z}_2 = \left\{\frac{-1}{2(1+\cos(t\pi))}: t \in (0,1)\cap\mathbb{Q}\right\}.
\]

We will now prove the final lemma needed to determine $\mathcal{A}_2$.

\begin{lemma}
    \label{lem: N_2^1}
    The family $\mathcal{R}_2^1$ is not normal around any $\lambda \in (-\infty, -1/4]$, 
    i.e. $(-\infty, -1/4] \subseteq \mathcal{A}_2^1$.
\end{lemma}
\begin{proof}
Recall that 
\[
    \mathcal{R}_2^1 = \{R_{P_n,v_n}: n \geq 1\} = \{\lambda \mapsto f_\lambda^n(0): n \geq 1\}.
\]
Take a $\lambda_0 \in (-\infty, -1/4]$ and suppose for the sake of contradiction 
that there exists a neighborhood $U$ of $\lambda_0$ on which $\mathcal{R}_2^1$ is normal.
We take $U$ connected and sufficiently small so that it does not contain $0$, and
$0$ is not a fixed point of $f_\lambda$ for any $\lambda \in U$. Because $\mathcal{R}_2^1$
is normal on $U$ there exists a subsequence of $\{R_{P_n,v_n}\}_{n \geq 1}$ that converges
locally uniformly to a holomorphic function $F: U \to \Chat$. 
For $\lambda \in U \setminus (-\infty, -1/4]$ the map $f_\lambda$ is loxodromic, hence
$f_\lambda^n(0)$ converges to the attracting fixed point of $f_\lambda$ as $n$ goes to
infinity. This means that for $\lambda \in U \setminus (-\infty, -1/4]$ we have 
$f_\lambda(F(\lambda)) = F(\lambda)$. Because $U \setminus (-\infty, -1/4]$ is non-empty and
open in $U$ it follows from the identity theorem for holomorphic functions that 
the equality $f_\lambda(F(\lambda)) = F(\lambda)$ must hold for all $\lambda \in U$.
The set $U$ contains a parameter $\lambda_1$ for which $f_{\lambda_1}$ is elliptic. The value
$0$ is not a fixed point of $f_{\lambda_1}$ and thus the distance of $f_{\lambda_1}^n(0)$ to
both of the fixed points of $f_{\lambda_1}$ is uniformly bounded below for all $n$ by
a positive constant. This means that no subsequence of $\{R_{P_n,v_n}(\lambda_1)\}_{n \geq 1}$
can converge to the fixed point $F(\lambda_1)$. We conclude that $\mathcal{R}_2^1$ is not
normal at $\lambda_0$.
\end{proof}

It follows from the previous two lemmas and Corollary~\ref{cor:zerofree->normal}
that 
\[
    (-\infty, -1/4] \subseteq \mathcal{A}_2^1\subseteq \mathcal{A}_2^2 \subseteq \overline{\mathcal{Z}_2} = (-\infty, -1/4].
\]
Therefore we can conclude that both $\mathcal{A}_2$
and $\overline{\mathcal{Z}_2}$ are equal to $(-\infty, -1/4]$.

\subsection{Determining the density-locus.}

Recall that for $\lambda \in \mathbb{C}$ we defined $\mathcal{R}_\Delta^i(\lambda) 
= \{R_{G,v}(\lambda): (G,v) \in \mathcal{G}_\Delta^i\}$. Subsequently we defined
$\mathcal{D}_\Delta^i$ as the set consisting of those $\lambda$ for which 
$\mathcal{R}_{\Delta}^i(\lambda)$ is dense in $\Chat$. It is thus clear that
$\mathcal{D}_2^1 \subseteq \mathcal{D}_2^2$. 
To conclude the section we show the following.
\begin{lemma}
    There is no $\lambda_0 \in \mathbb{C}$ for which $\mathcal{R}_2^2(\lambda_0)$ 
    is dense in $\Chat$, i.e. $\mathcal{D}_2^2 = \emptyset$.
\end{lemma}
\begin{proof}
It follows from Theorem~\ref{thm: stable paths} that $$
\mathcal{R}_2^2(\lambda)
=\{R_{T,v}(\lambda): (T,v) \in \mathcal{G}_2^2 \text{ with $T$ a tree}\}.
$$
A rooted tree $(T,v) \in \mathcal{G}_2^2$ can be viewed as a vertex $v$ onto which
two rooted paths $(P_n,v_n)$ and $(P_m, v_m)$ are attached for $n,m \geq 0$. It 
follows from Lemma~\ref{lem: tree recursion} and Lemma~\ref{lemma: paths} that
\[
    R_{T,v}(\lambda) =  \lambda \cdot \frac{1}{1+f_\lambda^n(0)} 
    \cdot \frac{1}{1+f_\lambda^m(0)} = \frac{1}{\lambda} \cdot f_\lambda^{n+1}(0) \cdot f_\lambda^{m+1}(0).
\]
For a specific $\lambda_0$ the right-hand side of this equality is not defined if 
$f_{\lambda_0}^{n+1}(0)$ and $f_{\lambda_0}^{m+1}(0)$ take on the values $0$ and 
$\infty$. Recall that if $f_{\lambda_0}^{n+1}(0)= \infty$, then $f_{\lambda_0}^{n}(0)= -1$,
which implies that $\lambda_0 \in \mathcal{Z}_2$. If this is the case then Lemma \ref{lem:zerosDelta2} implies that $\lambda_0$ is real,
and thus $\mathcal{R}_2^2(\lambda_0)$ is contained in $\mathbb{R}\cup \{\infty\}$ and is not
dense in $\Chat$. 

Assume that $\lambda_0$ is not real. In this case we have the equality 
\[
\mathcal{R}_2^2(\lambda_0) = \left\{\frac{1}{\lambda_0} \cdot f_{\lambda_0}^{n+1}(0) \cdot f_{\lambda_0}^{m+1}(0): n,m \geq 0\right\}.
\]
The map $f_{\lambda_0}$ is loxodromic, hence the orbit of $0$ converges to an attracting
fixed point without passing through $\infty$. Note that $f_{\lambda_0}(\infty) = 0$, therefore
$\infty$ is not the attracting fixed point, and thus there is a positive bound 
$B \in \mathbb{R}_{>0}$ such that $|f_{\lambda_0}^{n}(0)| < B$ for all $n$. It follows that 
\[
    \left|\frac{1}{\lambda_0} \cdot f_{\lambda_0}^{n+1}(0) \cdot f_{\lambda_0}^{m+1}(0)\right|
    < \frac{B^2}{|\lambda_0|}
\]
for all $n,m$, and thus $\mathcal{R}_2^2(\lambda_0)$ is bounded and in particular not dense in $\Chat$.
\end{proof}

\section{Equality of the zero-locus and the activity-locus for $\Delta\geq 3$}
\label{sec: Ndelta = zdelta}
In this section we prove the equalities $\mathcal{A}_\Delta^1 = \mathcal{A}_\Delta^2 = \cdots = \mathcal{A}_\Delta^\Delta = \overline{\mathcal{Z}_\Delta}$ for $\Delta \geq 3$, 
thereby proving that the activity-locus is equal to
the zero-locus. 
Our strategy is similar to the $\Delta = 2$ case. 
By definition we have $\mathcal{A}_\Delta^1 \subseteq \mathcal{A}_\Delta^2 \subseteq \cdots \subseteq
\mathcal{A}_\Delta^{\Delta}$.
We will first show that $\mathcal{A}_\Delta^1 = \mathcal{A}_\Delta^2 = \cdots =
\mathcal{A}_\Delta^{\Delta-1}$ and subsequently we will show that
$\overline{\mathcal{Z}_\Delta} \subseteq \mathcal{A}_\Delta^{\Delta-1}$. 
Then Corollary~\ref{cor:zerofree->normal}, which states that 
$\mathcal{A}_\Delta^\Delta \subseteq \overline{\mathcal{Z}_\Delta}$, is enough to
arrive at our desired conclusion.

\begin{lemma}\label{lem:Ndeltawelldefined}
    The family $\mathcal{R}_\Delta^{1}$ is normal at $\lambda_0 \in \mathbb{C}$
    if and only if $\mathcal{R}_\Delta^{\Delta-1}$ is normal at $\lambda_0$, and hence $\mathcal{A}_\Delta^{1} = \mathcal{A}_\Delta^{\Delta-1}$.
\end{lemma}

\begin{proof}
    Recall that 
    \[
        \mathcal{R}_\Delta^i := \{R_{G,v}: (G,v) \in \mathcal{G}_\Delta^i\}
    \]
    and thus $\mathcal{R}_\Delta^1 \subseteq \mathcal{R}_\Delta^{\Delta-1}$. It follows that 
    if $\mathcal{R}_\Delta^{\Delta-1}$ is normal at $\lambda_0$ then the same holds for $\mathcal{R}_\Delta^{1}$.
    
    To show the other direction, assume that $\mathcal{R}_\Delta^{1}$ is normal at $\lambda_0$. 
    Note that the family $\mathcal{R}_\Delta^\Delta$ is normal at $0$
    by Lemma~\ref{lem: RegionAroundlambda=0}, hence we can assume $\lambda_0 \neq 0$. As $\mathcal{R}_\Delta^{1}$ is normal at $\lambda_0$, there is a neighborhood $U$ of $\lambda_0$ on which 
    $\mathcal{R}_\Delta^{1}$ is a normal family. We can take $U$ such that $0 \not \in U$.
    We will show that $\mathcal{R}_\Delta^{\Delta-1}$ is 
    also a normal family on $U$. To that effect take a sequence of rooted graphs
    $\{(G_n,v_n)\}_{n \geq 1} \subseteq \mathcal{G}_\Delta^{\Delta-1}$.
    Construct the rooted graphs $(\hat{G}_n, \hat{v}_n) \in \mathcal{G}_\Delta^1$
    by attaching a root $\hat{v}_n$ to the root $v_n$ of $G_n$ by a single edge.
    By assumption the sequence $\{R_{\hat{G}_n,\hat{v}_n}\}_{n \geq 1}$ has 
    a subsequence that converges locally uniformly to a function $H: U \to \Chat$. 
    Let $I \subseteq \mathbb{N}$ be the indices belonging to this subsequence.
    By Lemma~\ref{lemma: paths} we have $R_{\hat{G}_n,\hat{v}_n}(\lambda) 
    = f_\lambda(R_{G_n,v_n}(\lambda))$ for every $\lambda \in U$.
    Because $U$ does not contain $0$, the M\"obius transformation $f_\lambda$ is 
    invertible for every $\lambda\in U$. Therefore for these $\lambda$ we have
    \[
        \lim_{\substack{n \to \infty\\n \in I}} R_{G_n,v_n}(\lambda)
        =
        \lim_{\substack{n \to \infty\\n \in I}} 
        f_\lambda^{-1}(f_\lambda(R_{G_n,v_n}(\lambda)))
        =
        \lim_{\substack{n \to \infty\\n \in I}} 
        f_\lambda^{-1}(R_{\hat{G}_n,\hat{v}_n}(\lambda))
        =
        f_\lambda^{-1}(H(\lambda)).
    \]
    Because the map $U \times \Chat \to \Chat$ that sends $(\lambda,z)$ to $f_\lambda^{-1}(z)$
    is continuous, we can conclude that this limit converges locally uniformly on
    $U$. Therefore we have shown that the sequence $\{R_{G_n,v_n}\}_{n \geq 1}$ has a
    subsequence that converges locally uniformly to the holomorphic function 
    $\lambda \mapsto f_\lambda^{-1}(H(\lambda))$, and thus $\mathcal{R}_\Delta^{\Delta-1}$ is normal
    at $\lambda_0$.
\end{proof}

\begin{prop}\label{normal->zerofree}
Let $\Delta \geq 3$. Then $\overline{\mathcal{Z}_\Delta} \subseteq \mathcal{A}_\Delta^{\Delta-1}$, and hence the zero-locus is contained in the activity-locus.
\end{prop}
\begin{proof}
Let us assume $\lambda \in \overline{\mathcal{Z}_\Delta}$. Then for any open neighborhood $V$ of $\lambda$ there is a $\lambda_0 \in V$ for which $Z_G(\lambda_0) = 0$ for some $G \in \mathcal{G}_\Delta$. We will prove that the family $\mathcal{R}_\Delta^{\Delta-1}$ cannot be normal on $V$.

By Lemma~\ref{lem: zero implies -1 at leaf} there is a a rooted tree $(T,u) \in \mathcal{G}_\Delta^1$ for which $R_{T,u}(\lambda_0) = -1$. 
Consider the rooted trees $(T_n,v)$ obtained by implementing a copy of $(T,u)$ in every vertex of the rooted paths $(P_n, v)$. It follows from Lemma~\ref{cor:implementing} that
$$
R_{T_n, v} = R_{P_n,v} \circ R_{T,u}.
$$
We note that in $T_n$ the root $v$ has degree $2 \leq \Delta - 1$. Furthermore $R_{T,u}$ maps a neighborhood of $\lambda_0$ holomorphically to a neighborhood of $-1$, since $R_{T,u}$ is not constantly equal to $-1$. Lemma~\ref{lem: N_2^1} states that the family $\{R_{P_n, v}\}_{n > 0}$ is not normal at $-1$ and thus it follows that $\{R_{T_n,v}\}_{n > 0}$ is not normal at $\lambda_0$.
\end{proof}

Summarising we have the following relations between sets
\[
\mathcal{A}_\Delta^1 \labelrel={eq:Ndeltawelldefined} \mathcal{A}_\Delta^{\Delta-1} \subseteq \mathcal{A}_\Delta^\Delta \labelrel\subseteq{eq:corzerofree->normal} \overline{\mathcal{Z}_\Delta} \labelrel\subseteq{eq:normal->zerofree} \mathcal{A}_\Delta^{\Delta-1},
\]
where equality~\eqref{eq:Ndeltawelldefined} is due to Lemma~\ref{lem:Ndeltawelldefined}, inclusion \eqref{eq:corzerofree->normal} is due to Corollary~\ref{cor:zerofree->normal} and inclusion \eqref{eq:normal->zerofree} is due to Proposition~\ref{normal->zerofree}.
It follows that $\mathcal{A}_\Delta^1 = \ldots = \mathcal{A}_\Delta^\Delta$, and $\mathcal{A}_\Delta = \overline{\mathcal{Z}_\Delta}$ for all $\Delta \geq 2$.

\subsection{The complement of the zero-locus}\label{sec: opposite}
As an application of the equality of the zero-locus and the activity-locus, we show here that each component of the complement of the zero-locus is simply connected.
We recall from the introduction that this implies that if the complement of the zero-locus is connected (as we conjecture in Conjecture~\ref{conj:connected}), then our main result gives a complete understanding of the complexity of approximately computing the independence polynomial.

\begin{prop}\label{prop:simply connected}
Let $\Delta\geq 2$ be an integer. Any connected component of the complement of the zero-locus,  $\mathbb{C}\setminus \overline{\mathcal{Z}_\Delta}$, is simply connected.
\end{prop}
\begin{proof}
For $\Delta = 2$ the statement follows directly by the exact characterization of the closure of the zero-locus. We will therefore assume that $\Delta \ge 3$.

Let $\gamma$ be a simple closed curve contained in the complement of the zero-locus, $\mathbb{C}\setminus \overline{\mathcal{Z}_\Delta}$. It is sufficient to prove that the interior of $\gamma$, which we will denote by $V$, is zero free. Let us suppose for the sake of a contradiction that this is not the case.

Let $T$ be a minimal tree for which $Z_T(\lambda_0) = 0$ for some $\lambda_0 \in V$. Let $v$ be a leaf of $T$. 
Since $|T|$ is chosen minimal it follows that $R_{T,v}(\lambda_0) = -1$. Denote the neighbor of $v$ in $T$ by $w$. 
By minimality of $T$ it also follows that $R_{T-v,w}(\lambda) \neq -1$ for any $\lambda \in \overline{V}$.

Note that $V$ is necessarily bounded, as it is a subset of the cardioid, $\Lambda_\Delta$. 
Hence by compactness of $\overline{V}$ it follows that $R_{T-v,w}$ is bounded away from $-1$ on $V$. 
Since
$$
R_{T,v}(\lambda) = \frac{\lambda}{1+R_{T-v,w}(\lambda)}
$$
it follows that $R_{T,v}$ is bounded on $\overline{V}$. By the Open Mapping Theorem for holomorphic functions it follows that there must be a $\lambda_1 \in \partial V = \gamma$ for which $$
R_{T,v}(\lambda_1) \in (-\infty, -1).
$$

Use Lemma~\ref{lemma: paths} to implement the rooted tree $(T,v)$ in the paths $P_n$ to obtain a sequence of rooted graphs $\{(G_n,u_n)\}_{n\geq 1}$ with
$$
R_{G_n,u_n} = R_{P_n,u_n}\circ R_{T,v}.
$$
Since $R_{T,v}(\lambda_1) \in (-\infty, -1)$, which is contained in the half-line where the family $\{\lambda \mapsto R_{P_n,u}(\lambda)\}_{n \in \mathbb N}$ is not normal, it follows that the family of ratios $\{R_{G_n,u_n}\}$ is not normal at $\lambda_1$. This contradicts the assumption that $\gamma$ is contained in $\mathbb{C}\setminus \overline{\mathcal{Z}_\Delta}$, by the equivalence of the activity-locus and the zero-locus.
\end{proof}

\section{Equality of the density-locus and the activity-locus for $\Delta \geq 3$}

We first show the inclusion $\overline{\mathcal{D}_\Delta} \subseteq \mathcal{A}_\Delta$ holds. Note that as $\mathcal{A}_\Delta$ is closed, it suffices to show $\mathcal{D}_\Delta \subseteq \mathcal{A}_\Delta$.

\begin{theorem}\label{thm:dense->nonnormal}
The density-locus is contained in the activity-locus. More precisely, we have $\mathcal{D}_\Delta \subseteq \mathcal{A}_\Delta$ for all $\Delta \geq 3$.
\end{theorem}

\begin{remark} Recall the remarkable Proposition 6 in \cite{Galanisetal20}, in which it is shown that non-real $\lambda \in \mathbb Q[i]$ outside the cardioid $\Lambda_\Delta$ are contained in the density-locus. As a consequence Theorem \ref{thm:dense->nonnormal} implies that $\mathcal{Z}_\Delta$ is dense in the complement of the cardioid.
\end{remark}

The proof of Theorem~\ref{thm:dense->nonnormal} is by contradiction. So we will assume that there is a $\lambda_0 \in \mathcal{D}_\Delta$ with $\lambda_0 \not \in \mathcal{A}_\Delta$ and arrive at a contradiction. In order to do this, we state and prove three helpful lemmas. 

\begin{lemma}\label{contradictionlemma1}
Let $\lambda_0 \in \mathbb C \setminus \mathcal{A}_\Delta$. Assume the family $\mathcal{R}_\Delta$ is normal on some open neighborhood $U$ of $\lambda_0$ and that $\{R_{G_n,v_n}(\lambda_0)\}_{n \geq 1}$ converges to $-1$ for a sequence $\{(G_n,v_n)\}_{n\geq 1}$ of rooted graphs from $\mathcal G_\Delta$.
Then $\{R_{G_n,v_n}\}_{n\geq 1}$ converges to $-1$ locally uniformly on $U$.
\end{lemma}
\begin{proof}
It follows from the conclusion of Section \ref{sec: Ndelta = zdelta}, i.e. $\mathcal{A}_\Delta = \overline{\mathcal{Z}_\Delta}$, that $Z_G(\lambda) \neq 0$ for all $\lambda \in U$ and $G \in \mathcal{G}_\Delta$. 
Suppose $\{R_{G_n,v_n}\}_{n\geq 1}$ does not converge to $-1$ locally uniformly on $U$. 
Then, after taking a subsequence if necessary, we may assume that $\{R_{G_n,v_n}\}_{n\geq 1}$ converges locally uniformly on $U$ to a non-constant holomorphic function $f$.
Clearly $f(\lambda_0)=-1$. Since zeros of holomorphic functions are isolated there exists $\varepsilon>0$ so that $\overline{B(\lambda_0,\varepsilon)}\subset U$ and such that 
$$
\delta := \inf_{\lambda\in\partial B(\lambda_0,\varepsilon)}|f(\lambda)+1|>0.
$$
Let $n$ be sufficiently large so that $|R_{G_n,v_n}-f|<\delta$ uniformly on $\overline{B(\lambda_0,\varepsilon)}$.
Then 
\[|(R_{G_n,v_n}(\lambda)+1)-(f(\lambda)+1)|<\delta< |f(\lambda)+1| + |R_{G_n,v_n}(\lambda)+1|
\]
for all $\lambda \in \partial B(\lambda_0,\varepsilon)$. 
By Rouch\'e's theorem there exists $\lambda_1\in B(\lambda_0,\varepsilon)$ for which $R_{G_n,v_n}(\lambda_1)=-1$.
By Lemma \ref{lemma:equivalence} it follows $\lambda_1$ is a zero of the independence polynomial $Z_{G}$ for some graph $G$ of maximum degree at most $\Delta$, which is a contradiction as we assumed $\lambda_0 \in \mathbb C \setminus \mathcal{A}_\Delta = \mathbb C \setminus \overline{\mathcal{Z}_\Delta}$. 
\end{proof}

\begin{lemma}\label{contradictionlemma2}
Let $\lambda_0 \in \mathbb C \setminus \mathcal{A}_\Delta$. Assume the family $\mathcal{R}_\Delta$ is normal on some open neighborhood $U$ of $\lambda_0$ and that $\{R_{G_n,v_n}(\lambda_0)\}_{n\geq 1}$ converges to $\mu\le -\frac{1}{4}$ for a sequence of rooted graphs $\{(G_n,v_n)\}_{n\geq 1}$ in $\mathcal G_\Delta^1$.
Then $\{R_{G_n,v_n}\}_{n\geq 1}$ converges to $\mu$ locally uniformly on $U$.
\end{lemma}
\begin{proof}
If this is not the case then, as in the previous lemma, we may assume that $\{R_{G_n,v_n}\}_{n\geq 1}$ converges locally uniformly to a non-constant holomorphic function $f$ with $f(\lambda_0)=\mu$. 
By Rouch\'e's theorem we can find $\lambda_1\in U$ and $n$ sufficiently large so that $R_{G_n,v_n}(\lambda_1)=\mu$, by the same argument as in the previous lemma.

Consider the family of rooted graphs $\{(\widetilde{G}_k,w_k)\}$ obtained by implementing $(G_n,v_n)$ in every vertex of the rooted paths $(P_{k},w_k)$, where $P_k$ is the path with $k$ vertices and $w_k$ is one of its extreme vertices. 
Since $v_n$ has degree $1$, the graph $\widetilde{G}_k$ has maximum degree at most $\Delta$. 
Hence by Lemma~\ref{lemma: paths} we have
$$
R_{\widetilde{G}_k,w_k}(\lambda)=f^{k}_{R_{G_n,v_n}(\lambda)}(0).
$$
By Lemma~\ref{lem: N_2^1} the family $\{R_{P_k,w_k}\}=\{\lambda\mapsto f^k_{\lambda}(0)\}$ is non-normal at $\lambda=\mu$, and therefore the family  $\{R_{\widetilde{G}_k,w_k}\}$ is non-normal at $\lambda_1$, contradicting the fact that the family $\mathcal{R}_\Delta$ is normal on $U$.
\end{proof}


\begin{lemma}\label{contradictionlemma3}
Assume there is a $\lambda_0 \in \mathcal{D}_\Delta$ with $\lambda_0 \not \in \mathcal{A}_\Delta$. Denote $U$ for an open neighborhood of $\lambda_0$ on which the family $\mathcal{R}_\Delta$ is normal.
Assume furthermore that $\{R_{G_n,v_n}(\lambda_0)\}_{n\geq 1}$ converges to $\mu\in \mathbb R$ for a sequence of rooted graphs  $\{(G_n,v_n)\}_{n\geq 1}$ in $\mathcal G_\Delta^{\Delta-1}$.
Then $\{R_{G_n,v_n}\}_{n\geq 1}$ converges to $\mu$ locally uniformly on $U$.
\end{lemma}
\begin{proof}
If $\mu=-1$ the result follows by Lemma~\ref{contradictionlemma1}. 
We may therefore assume that $\mu\neq -1$. Recall that we denote $f_\lambda(z)=\lambda/(1+z)$. We will show for each $\mu \in \mathbb{R}$ there exists $\mu_1, \mu_2, \mu_3 \le -\frac{1}{4}$ so that
$$
f_{\mu_m}\circ\cdots\circ f_{\mu_1}(\mu)=-1,
$$
for some $m \leq 3$.
We distinguish between different cases
\begin{enumerate}
    \item $\mu\ge -3/4$. Take $\mu_1=-1-\mu\le -1/4$, one can check $f_{\mu_1}(\mu)=-1$.
    \item $\mu<-1$. Take $\mu_1=-1/4$ and $\mu_2 = 1 - f_{\mu_1}(\mu)$, then $f_{\mu_1}(\mu)>0>-3/4$ and so $\mu_2 \leq -1/4$. One can check $f_{\mu_2} \circ f_{\mu_1} (\mu) = -1$.
    \item $-1<\mu<-3/4$. Take $\mu_1= \mu_2 = -1/4$ and $\mu_3 = 1 - f_{\mu_2}(f_{\mu_1}(\mu))$, then $f_{\mu_1}(\mu)<-1$ so we see $\mu_3 \leq -1/4$. One can check $f_{\mu_3} \circ f_{\mu_2} \circ f_{\mu_1} (\mu) = -1$.
\end{enumerate}

We may assume that $\{R_{G_n,v_n}\}_{n\geq 1}$ converges locally uniformly on $U$ to a holomorphic function $f$ with $f(\lambda_0)=\mu$. We want to show $f$ is constant on $U$.
Since the set $\{R_{G,v}(\lambda_0): (G,v) \in \mathcal{G}_\Delta^1\}$ is dense in $\hat{\mathbb{C}}$ by assumption, we can choose sequences of rooted graphs $\{(G_n^i,v_n^i)\}_{n \geq 1}$ in $\mathcal{G}_\Delta^1$ so that $\{R_{G_n^i,v_n^i}(\lambda_0)\}_{n\geq 1}$ converges to $\mu_i$ for each $i=1,\ldots,m$.
By Lemma~\ref{contradictionlemma2} every sequence $\{R_{G_n^i,v_n^i}\}_{n\geq 1}$ converges locally uniformly on $U$ to the constant function $\mu_i$ for each $i$.

Consider for each $n\geq 1$, the rooted graph $(\widetilde G_n,v_n^m)$ obtained by implementing the rooted graphs $(G_n,v_n)$, $(G_n^1,v_n^1),\ldots, (G_n^m,v_n^m)$ on the vertices of the path $P_{m+1}$ of length $m$. 
Note that $\widetilde G_n$ has maximum degree at most $\Delta$.

It follows from Lemma~\ref{lemma: paths} that 
\[
R_{\widetilde G_n,v^m_n}(\lambda) = f_{R_{G_n^m,v^m_n}(\lambda)}\circ \cdots \circ f_{R_{G_n^1,v^1_n}(\lambda)}\circ R_{G_n,v_n}(\lambda).
\]
By our choice of the $\mu_i$ the sequence of ratios
$\{R_{\widetilde G_n,v^m_n}(\lambda_0)\}_{n\geq 1}$ converges to $f_{\mu_m}\circ\cdots\circ f_{\mu_1}(\mu)=-1$. Hence by Lemma \ref{contradictionlemma1} the sequence of ratios $\{R_{\widetilde G_n,v^m_n}\}_{n\geq 1}$ converges locally uniformly to the constant function $-1$.
Furthermore the sequence of ratios $\{R_{\widetilde G_n,v^m_n}\}_{n\geq 1}$ converges to the function $F:=f_{\mu_m}\circ \cdots \circ f_{\mu_1}\circ f$. As $f_{\mu_m}\circ \cdots \circ f_{\mu_1} (z) $ is a non-constant holomorphic function and $F = -1$ on $U$, it follows that $f$ is constant on $U$, as desired.
\end{proof}

We are now ready to prove Theorem \ref{thm:dense->nonnormal}.
\begin{proof}[Proof of Theorem~\ref{thm:dense->nonnormal}]
Assume for the purpose of a contradiction that there exists $\lambda_0 \in \mathcal{D}_\Delta^1$ with $\lambda_0 \not \in \mathcal{A}_\Delta$. We note that by Lemma~\ref{lem: RegionAroundlambda=0} we know $\lambda_0 \neq 0$. Throughout the proof denote $U$ for an open neighborhood of $\lambda_0$ on which the family $\mathcal{R}_\Delta$ is normal; we may assume $0 \not \in U$ by taking $U$ small enough. 
Assume first that $\lambda_0$ is not purely imaginary.  Consider the real number $c=\frac{- |\lambda_0|^2}{2\textrm{Re}\,\lambda_0}$ and notice that
$$
\frac{\lambda_0^2}{\lambda_0+c} = 2\textrm{Re}\,\lambda_0 \in\mathbb R.
$$
Choose two sequences of rooted graphs $\{G_n,v_n\}_{n\geq 1}$, $\{(H_n,w_n)\}_{n\geq 1}$ in  $\mathcal{G}_\Delta^1$ so that $\{R_{G_n,v_n}(\lambda_0)\}_{n\geq 1}$ and $\{R_{H_n,w_n}(\lambda_0)\}_{n\geq 1}$ converge to respectively $1$ and $c$.
By Lemma~\ref{contradictionlemma3} we must have that these sequence of ratios converge locally uniformly on $U$ to the respective constants $1$ and $c$.

Consider the sequence of graphs $\widetilde{G}_{n\geq 1}$ constructed by merging $v_n$ and $w_n$ and by then connecting this vertex to a vertex $\widetilde{v}_n$.

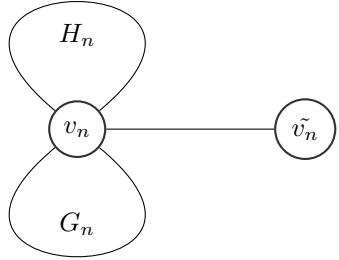
\begin{figure}[ht]

\centering

\begin{tikzpicture}[label distance=6mm]
\node[mycircle,label=below:{$G_n$}, label=above:{$H_n$}](a) {$v_n$};
\node[mycircle,right of=a,node distance=3cm](b) {$\tilde{v_n}$};

\draw[-] (a) to node {{}} (b);

\path[every node/.style={font=\sffamily\small}]
    (a) edge [loop,in=220,out=320,min distance=30mm] node {} (a)
    (a) edge [loop,in=40,out=140, min distance=30mm] node {} (a);

\path[every node/.style={font=\sffamily\small}]
    (b) edge [white,loop,in=220,out=320,min distance=30mm] node {} (b)
    (b) edge [white,loop,in=40,out=140, min distance=30mm] node {} (b);
\end{tikzpicture}

\caption{The rooted graph $\widetilde G_n$ in the proof of Theorem~\ref{thm:dense->nonnormal}}
    \label{graph2}
\end{figure}

It follows from Lemma~\ref{lemma: paths} and Lemma \ref{lem:mergingtrees} for all $\lambda \in U$ that
$$
R_{\widetilde G_n,\widetilde v_n}(\lambda)=\frac{\lambda}{1+\lambda^{-1}R_{G_n,v_n}(\lambda)\,R_{H_n,w_n}(\lambda)},
$$
where we use $0 \not \in U$.
Therefore the sequence of holomorphic functions $\{R_{\widetilde G_n,\widetilde v_n}\}_{n\geq 1}$ converges locally uniformly on $U$ to the function $f(\lambda)=\frac{\lambda^2}{\lambda+c}$ as $n\to \infty$.
Note that $f$ is not a constant function, and that $f(\lambda_0)\in\mathbb R$, contradicting Lemma~\ref{contradictionlemma3}.
This contradiction completes the proof for $\lambda_0$ not purely imaginary.

Assume instead that $\lambda_0$ is purely imaginary and let $(G,v)\in\mathcal{G}_\Delta^1$ so that $R_{G,v}(\lambda_0)$ is not purely imaginary. 
For $c\in \mathbb{R}$ to be determined later choose again two sequences of rooted graphs $\{G_n,v_n\}_{n\geq 1}$, $\{(H_n,w_n)\}_{n\geq 1}$ in $\mathcal{G}_\Delta^1$ such that sequences $\{R_{G_n,v_n}(\lambda_0)\}_{n\geq 1}$ and $\{R_{H_n,w_n}(\lambda_0)\}_{n\geq 1}$ converge to $1$ and $c$ respectively.
Define for each $n\geq 1$, $(\tilde{G_n},\widetilde{v_n})$ as above and let $(K_n,v_n)$ be the rooted graph obtained from the disjoint union of $(\tilde{G_n},\widetilde{v_n})$ and $(G,v)$ by identifying the vertex $\widetilde v_n$ with $v$. 
It follows from Lemma~\ref{lemma: paths} and Lemma \ref{lem:mergingtrees} for $\lambda \in U$ that
$$
R_{K_n,v_n}(\lambda)=\frac{R_{G,v}(\lambda)}{1+\lambda^{-1}R_{G_n,v_n}(\lambda)\,R_{H_n,w_n}(\lambda)},
$$
where we use $0 \not \in U$.
Thus in order to follow the same argument as before we require $c \in \mathbb R$ for which
$$
\frac{\lambda_0 \cdot R_{G,v}(\lambda_0)}{\lambda_0 + c} \in \mathbb R.
$$
It is clear that such real number $c$ exists, hence the identical argument leads to the desired contradiction. 
\end{proof}

We will now show the other inclusion $\mathcal{A}_\Delta \subseteq \overline{D_\Delta}$ also holds for all $\Delta \geq 3$.
We first show the inclusion holds for non-real parameters $\lambda \in \mathcal{A}_\Delta$.

\begin{theorem}\label{thm:nonnormalimpliesdense_realcase}
Let $\Delta \geq 3$ and suppose that the family $\mathcal{R}_\Delta$ is not normal in any neighborhood of $\lambda_0 \in \mathbb C \setminus \mathbb{R}_{\leq 0}$. Then there exists $\lambda_1$ arbitrarily close to $\lambda_0$ for which the set $\{R_{G,v}(\lambda_1) : (G,v) \in \mathcal{G}_\Delta^1\}$ is dense in $\hat{\mathbb C}$.
\end{theorem}

\begin{proof} 
Because $\overline{\mathcal{Z}_\Delta} = \mathcal{A}_\Delta$ there exists $\lambda_2$ arbitrarily close to $\lambda_0$ for which there is a graph $G$ of maximum degree at most $\Delta $ such that $Z_G(\lambda_2) = 0$. 
We claim that we can assume $\lambda_2 \not \in \mathbb{R}$. This is clear if $\lambda_0 \not \in \mathbb{R}$. Moreover, if 
$\lambda_0 \in \mathbb{R}$ then $\lambda_0$ is a strictly positive real number. Because $Z_G(x) > 0$ for 
any positive real number $x$, it follows that $\lambda_2$ is necessarily not real as long as it is sufficiently close to $\lambda_0$.

By Lemma \ref{lem: zero implies -1 at leaf} there is a rooted tree $(T,v) \in \mathcal{G}_\Delta^1$ such that $R_{T,v}(\lambda_2) = -1$. Since the rational function $\lambda \mapsto R_{T,v}(\lambda)$ is non-constant, it is an open map. The image of a neighborhood of $\lambda_2$ therefore contains a small open real interval around $-1$. Recall that Lemma~\ref{lem: f_lambda classification} states that for $\mu \in (-\infty, -1/4)$ the map $f_\mu: z \mapsto \mu/(1+z)$ is conjugate to a rotation $w \mapsto e^{i \theta} \cdot w$. Furthermore, by comparing $\tr^2$ of both maps, it is not hard to see that those parameters $\mu$ for which $f_\mu$ is conjugate to an irrational rotation lie dense in $(-\infty, -1/4)$. Therefore 
we can choose a $\lambda_1 \in \mathbb{C} \setminus \mathbb{R}$ arbitrarily close to $\lambda_2$ such that for $\mu:=R_{T,v}(\lambda_1)$
the map $f_\mu$ is conjugate to an irrational rotation. From now on $\mu$ is fixed to be this value.

Let ${p,q}$ be the two fixed points of the transformation $f_\mu$. In Section~\ref{sec: Mobius transformations} we explained that $\hat{\mathbb C} \setminus \{p,q\}$ is foliated by generalized circles invariant under $f_\mu$, and on which $f_\mu$ acts conjugate to an irrational rotation. We denote the generalized circle through $z$ by $C_z$, and write $C_q$ and $C_p$ for $\{q\}$ and $\{p\}$ respectively.
The map $z \mapsto C_z$ is continuous as a map from $\hat{\mathbb{C}}$ to the space $\{K \subseteq \hat{\mathbb C}: K \text{ compact}\}$ equipped with the Hausdorff metric.

Our goal is to show that $\mathcal{R}^1_\Delta(\lambda_1)$ is dense in $\hat{\mathbb{C}}$. 
We first claim that if $w \in \mathcal{R}^1_\Delta(\lambda_1)$, then $\mathcal{R}^2_\Delta(\lambda_1) \cap C_w$ is dense in $C_w$.

To prove the claim, let $(H,u) \in \mathcal{G}_\Delta^1$ be a rooted graph such that $R_{H,u}(\lambda_1) = w$. Let $\tilde{G}_n$ as follows be obtained from the path $P_{n+1}$ on $n+1$ vertices, labeled $v_0$ up to $v_{n}$, by implementing $(H,u)$ at $v_0$ and the rooted tree $(T,v)$ at the remaining $n$ vertices of $P_{n+1}$, see Figure \ref{graphclaim}.
Now by Lemma~\ref{lemma: paths} we have
\[
R_{\tilde{G}_n, v_{n}}(\lambda_1) = f_\mu^n(R_{H,u}(\lambda_1)) = f_\mu^n(w).
\]
Observe that for each $n\geq 1$ we have $(\tilde{G}_n, v_n)\in \mathcal{G}_\Delta^2$. 
Because $f_\mu$ acts conjugately to an irrational rotation on $C_w$ it follows that  $\mathcal{R}^2_\Delta(\lambda_1) \cap C_w$ is dense in $C_w$.

\begin{figure}[ht]
    \centering
\begin{tikzpicture}[label distance=6mm, every loop/.style={min distance=30mm,in=220,out=320}]
\node[mycircle,label=below:{$H$}](a) {$v_0$};
\node[mycircle,right of=a,node distance=2cm,label=below:{$T$}](b) {$v_1$};
\node[mycircle,right of=b,node distance=2cm,label=below:{$T$}](c) {$v_2$};
\node[mycircle,right of=c,node distance=3cm,label=below:{$T$}](d) {$v_n$};

\draw[-] (a) to node {{}} (b);

\draw[-] (b) to node {{}} (c);

\draw[dashed] (c) to node {{}} (d);

\path[every node/.style={font=\sffamily\small}]
    (a) edge [loop below] node {} (a)
    (b) edge [loop below] node {} (b)
    (c) edge [loop below] node {} (c)
    (d) edge [loop below] node {} (d);
   
\end{tikzpicture}
    \caption{The graph $(\tilde{G}_n, v_n)$ in the proof of the claim}
    \label{graphclaim}
\end{figure}
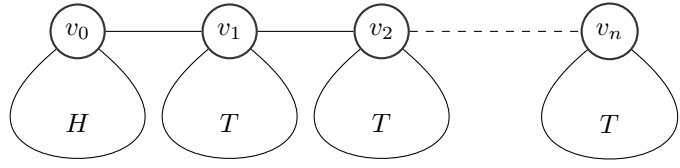

Because $\mu \in \mathcal{R}^1_\Delta(\lambda_1)$ and $C_\mu = \hat{\mathbb{R}} := \mathbb{R} \cup \{\infty\}$ it follows from the claim that $\mathcal{R}^2_\Delta(\lambda_1) \cap \hat{\mathbb{R}}$ is dense in $\hat{\mathbb{R}}$.
Observe that $f_{\lambda_1}(\hat{\mathbb{R}}) = \lambda_1 \cdot \hat{\mathbb{R}}$. So by attaching a vertex at the root with an edge, we obtain that $\mathcal{R}^1_\Delta(\lambda_1) \cap \lambda_1 \cdot \hat{\mathbb{R}}$ is dense in $\lambda_1 \cdot \hat{\mathbb{R}}$.

The set 
\[
    U = \{z \in \hat{\mathbb{C}}:\text{$C_z$ intersects $\lambda_1 \cdot \hat{\mathbb{R}}$ transversely}\}
\]
is an open set in $\hat{\mathbb{C}}$, see Figure \ref{fig:TransversalCircles}. Because $\lambda_1 \not \in \mathbb{R}$ we see that $C_{-1} = \hat{\mathbb{R}}$ intersects $\lambda_1 \cdot \hat{\mathbb{R}}$ transversely, and thus $-1 \in U$. 
The set U is contained in $\cup_{w \in \lambda_1 \cdot \hat{\mathbb{R}}} C_w$. 
Because $\mathcal{R}^1_\Delta(\lambda_1) \cap \lambda_1 \cdot \hat{\mathbb{R}}$ is dense in $\lambda_1 \cdot \hat{\mathbb{R}}$, it follows that $\cup_{w \in \mathcal{R}^1_\Delta(\lambda_1)} C_w$ is dense in $U$.
From the claim we proved earlier, it follows $\mathcal{R}^2_\Delta(\lambda_1)$ is dense in $U$. 
Attaching a vertex to the root of a tree in $\mathcal{G}_\Delta^2$ with ratio $r$ yields a rooted tree in $\mathcal{G}_\Delta^1$ with ratio $f_{\lambda_1}(r)$, and thus $\mathcal{R}^1_\Delta(\lambda_1)$ is dense in the neighborhood $U_\infty := f_{\lambda_1}(U)$ of $\infty$.

For two rooted trees $(T_1, v_1)\in\mathcal{G}_\Delta^1$ and $(T_2, v_2)\in\mathcal{G}_\Delta^2$ with ratios $r_1$ and $r_2$ respectively we can define the rooted tree $(T_3,v_1) \in \mathcal{G}_\Delta^2$ by adding an edge between the roots of $T_1$ and $T_2$ and considering $v_1$ the root of the obtained tree. By Lemma \ref{lemma: paths} the ratio of $(T_3,v_1)$ is given by
\[
    F(r_1,r_2):=f_{r_1}(r_2)=\frac{r_1}{1+r_2}
\]
under the assumption that this fraction is well defined, i.e., $(r_1,r_2) \not \in \{(0,-1),(\infty,\infty)\}$. It is not hard to see that 
\[
    F(U_\infty \times \hat{\mathbb{R}}\setminus \{(0,-1),(\infty,\infty)\}) = \hat{\mathbb{C}}.
\]
Because $\mathcal{R}^1_\Delta(\lambda_1)$ is dense in $U_\infty$ and $\mathcal{R}^2_\Delta(\lambda_1)$ is dense in $\hat{\mathbb{R}}$ it follows that $\mathcal{R}^2_\Delta(\lambda_1)$ is dense in $\hat{\mathbb{C}}$. We finally conclude that $\mathcal{R}^1_\Delta(\lambda_1)$ is dense in $f_{\lambda_1}(\hat{\mathbb{C}})=\hat{\mathbb{C}}$.
\end{proof}

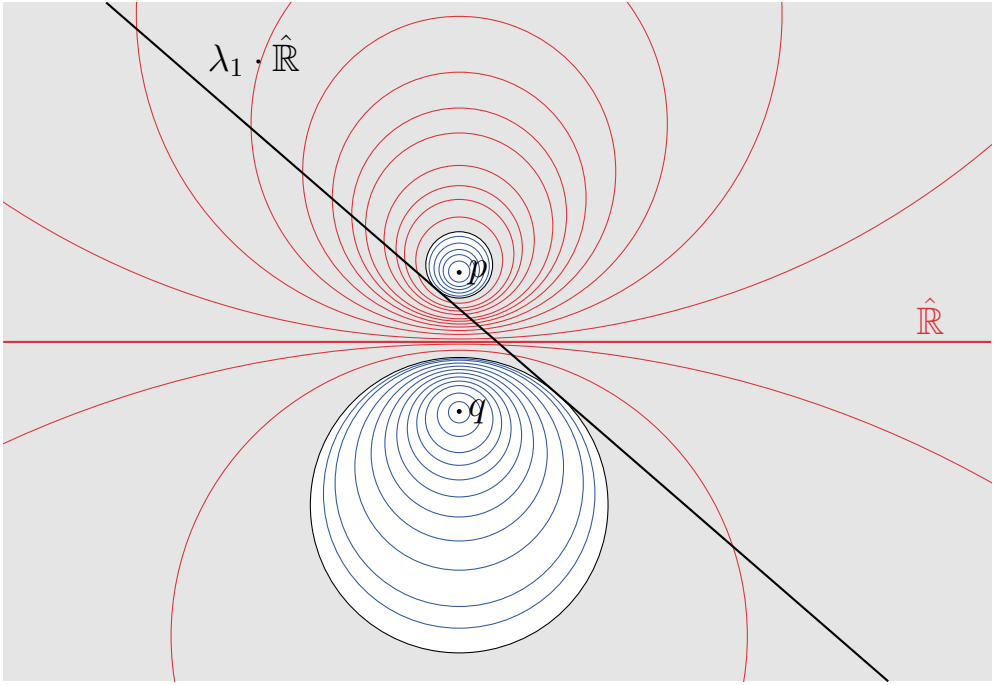
\begin{figure}[h!]
\centering
\begin{tikzpicture}
    
    \clip (-6.5,-4.5) rectangle (6.5,4.5);
    \fill[gray,opacity=0.2] (-6.5,-4.5) rectangle (6.5,4.5);
    
    \filldraw[black,fill=white] (-0.5, -2.16463) circle (1.95847);
    \filldraw[black,fill=white] (-0.5, 1.02139) circle (0.439597);
    
    \draw[gR] (-0.5, -3.90169) circle (3.7912);
    \draw[gR] (-0.5, -14.2842) circle (14.2544);
    \draw[gR] (-0.5, 10.8559) circle (10.8167);
    \draw[gR] (-0.5, 4.34695) circle (4.24805);
    \draw[gR] (-0.5, 2.88985) circle (2.73884);
    \draw[gR] (-0.5, 2.25668) circle (2.05975);
    \draw[gR] (-0.5, 1.90702) circle (1.66935);
    \draw[gR] (-0.5, 1.68766) circle (1.41358);
    \draw[gR] (-0.5, 1.5386) circle (1.23178);
    \draw[gR] (-0.5, 1.35156) circle (0.988289);
    \draw[gR] (-0.5, 1.24111) circle (0.830876);
    \draw[gR] (-0.5, 1.16958) circle (0.719667);
    \draw[gR] (-0.5, 1.08464) circle (0.571344);
    
    \draw[gB] (-0.5, -0.932707) circle (0.141217);
    \draw[gB] (-0.5, -0.96573) circle (0.287462);
    \draw[gB] (-0.5, -1.02346) circle (0.44438);
    \draw[gB] (-0.5, -1.07783) circle (0.558309);
    \draw[gB] (-0.5, -1.14721) circle (0.682714);
    \draw[gB] (-0.5, -1.23451) circle (0.82098);
    \draw[gB] (-0.5, -1.34375) circle (0.977578);
    \draw[gB] (-0.5, -1.48067) circle (1.15862);
    \draw[gB] (-0.5, -1.65362) circle (1.37275);
    \draw[gB] (-0.5, -1.87509) circle (1.63278);
    \draw[gB] (-0.5, -2.00973) circle (1.78578);
    
    \draw[gB] (-0.5, 1.00337) circle (0.395928);
    \draw[gB] (-0.5, 0.979527) circle (0.330867);
    \draw[gB] (-0.5, 0.959562) circle (0.266006);
    \draw[gB] (-0.5, 0.945848) circle (0.211254);
    \draw[gB] (-0.5, 0.932498) circle (0.139829);
    
    \draw[gR, thick] (-6.5, 0) -- (6.5, 0);
    \draw[black,thick] (5.14456, -4.5) -- (-5.14456, 4.5);
    
    \filldraw[black] (-0.5, -0.921954) circle (0.025) node[anchor=west] {\LARGE$q$};
    \filldraw[black] (-0.5, 0.921954) circle (0.025) node[anchor=west] {\LARGE$p$};
    \node[gR] at (5.7,0.35) {\LARGE$\hat{\mathbb{R}}$};
    \node[black] at (-3.2,3.8) {\LARGE$\lambda_1 \cdot \hat{\mathbb{R}}$};
\end{tikzpicture}
\caption{The generalized circles $\lambda_1 \cdot \hat{\mathbb{R}}$ and $\hat{\mathbb{R}}$ intersect in the points $0$ and $\infty$. A region around $0$ is drawn. The open set $U$ is shaded in gray. Examples of generalized circles $C_w$ that intersect $\lambda_1 \cdot \hat{\mathbb{R}}$ transversely are drawn in red, while examples of circles that do not intersect $\lambda_1 \cdot \hat{\mathbb{R}}$ are drawn in blue.}
\label{fig:TransversalCircles}
\end{figure}

We can now finally prove the inclusion $\mathcal{A}_\Delta \subseteq \overline{\mathcal{D}_\Delta}$ building on Proposition 6 of \cite{Galanisetal20} to deal with the real parameters $\lambda\in \mathcal{A}_\Delta$.

\begin{theorem}
Let $\Delta \geq 3$. Then the activity locus is contained in the density locus, i.e. $\mathcal{A}_\Delta \subseteq \overline{\mathcal{D}_\Delta}$.
\end{theorem}
\begin{proof}
Let $\lambda_0 \in \mathcal{A}_\Delta$. If $\lambda_0 \in \mathbb C \setminus \mathbb{R}_{\leq 0}$, then $\lambda_0 \in \overline{\mathcal{D}_\Delta}$ follows from Theorem \ref{thm:nonnormalimpliesdense_realcase}.

We know $\overline{\mathcal{Z}_\Delta} = \mathcal{A}_\Delta$. 
By Remark \ref{rmk:shearerandcardioid} we know that 
$$
\overline{\mathcal{Z}_\Delta} \cap \mathbb{R}_{\leq 0}= \mathbb{R}_{\leq 0} \setminus \Int(\Lambda_\Delta)
$$
for all $\Delta \geq 3$. 
Proposition 6 of \cite{Galanisetal20} implies that $\mathbb{C} \setminus (\mathbb{R}\cup \Lambda_\Delta) \subseteq \overline{\mathcal{D}_\Delta}$ for $\Delta \geq 3$.
Hence it follows that $\mathbb{R}_{\leq 0} \setminus \Int(\Lambda_\Delta) \subseteq \overline{\mathcal{D}_\Delta}$, which completes the proof.
\end{proof}

\section{Density implies \#P-hardness}

In this section we will show that the density-locus is contained in the $\#\mathcal{P}$-locus. To prove our result we will need to to show `exponential' density for ratios of a specific family of trees: we need to get $\varepsilon$-close to a given point $P\in \CQ$ with ratios of trees of size at most $O(\log(1/\varepsilon)+\size P)$. Here $\size P$ denotes the sum of the bit sizes of the real and imaginary part of $P$. Moreover, we denote for rational $\varepsilon>0$ by $\size{\varepsilon,P}$ the sum of the the bit size of $\varepsilon$ and $\size{P}$. 

Let $\lambda_0 \in \mathbb{C}\setminus \mathbb{R}$. 
Then the M\"obius transformation $f_{\lambda_0}$ is loxodromic (cf. Section~\ref{sec: Mobius transformations}) and hence has a repelling fixed point, which we denote by $z_0$.

Let \begin{equation}\label{eq:A}
    A := \{z \in \mathbb{C}: 2\pi/3-0.01 < \arg{z} < 2\pi/3 , 1/17 < |z| < 1/16\}.
\end{equation}
Let $\mathcal{T}=\{(G_1,v_1),\ldots, (G_m,v_m),(\overline{G}_1,\overline{v}_1),\ldots,(\overline{G}_M,\overline{v}_M\}$ be a family of rooted trees and $U$ an open disk containing $z_0$. 
The pair $(\mathcal{T},U)$ is called a \emph{fast implementer for $\lambda_0$} if the ratios $\mu_i:=R_{G_i,v_i}$ and $\chi_i:=R_{\overline{G}_i,\overline{v}_i}$ are such that the maps $g_i:=f_{\mu_i} \circ f_{\chi_i}$ are loxodromic and satisfy
\begin{enumerate}
    \item the attracting fixed point $z_i$ of $g_i$ lies in $U$ for all $i$,
    \item $\overline{U} \subseteq \cup_{i=1}^M g_i(U)$,
    \item $g_i'(z) \in A$ for all $i$ and all $z \in \overline{U}$,
\end{enumerate}
and the disk $U$ is such that
\begin{enumerate}
    \item $\overline{U}\subset f_{\lambda_0}(U)$,
    \item $\overline{U}$ does not contain the attracting fixed point of $f_{\lambda_0}$,
    \item $U$ has three rational points on its boundary.
\end{enumerate}

We have the following results concerning fast implementers.

\begin{lemma}
\label{lem: fast implementers exist}
Let $\Delta\geq 3$ be an integer.
Let $\lambda_0\in \mathcal{D}_\Delta$.
Then there exists a fast implementer $(\mathcal{T},U)$ for $\lambda_0$.
\end{lemma}

\begin{lemma}\label{alg: fast implementation}
Let $\lambda_0\in \mathbb{C}\setminus \mathbb{R}$ and assume that there exists a fast implementer
$(\{(G_1,v_1),\ldots, (G_m,v_m),(\overline{G}_1,\overline{v_1}),\ldots,(\overline{G}_M,\overline{v}_M)\},U)$ for $\lambda_0$.
Then, given $P \in \mathbb{C}$ and $\epsilon > 0$ there exists an algorithm that
yields a sequence of ratios
\[
    w_1, \dots, w_K \in \{\lambda_0\} \cup \bigcup_{i=1}^M \{\mu_i:=R_{G_i,v_i},\chi_i:=R_{\overline{G}_i,\overline{v_i}}\}
\]
such that $|(f_{w_K} \circ \cdots \circ f_{w_1})(0) - P| < \epsilon$, $w_K = \lambda_0$ 
and $K = \mathcal{O}(\max{(\log(1/\epsilon),\log(|P|/\epsilon))})$. 
If $\lambda_0 \in \CQ$ and the input parameters $P,\epsilon$ are also 
in $\CQ$ then the algorithm runs in $poly(\size{P,\epsilon})$ time.
\end{lemma}

We provide proofs for these lemmas in the next subsection, but first we collect some consequences.

\begin{corollary}\label{cor:density-locus is open}
Let $\Delta\geq 3$ be an integer. The set $\mathcal{D}_\Delta$ is an open set.    
\end{corollary}
\begin{proof}
Let $\lambda_0\in \mathcal{D}_\Delta.$ Let $(\mathcal{T},U)$ be fast implementer, as guaranteed to exist by Lemma~\ref{lem: fast implementers exist}.
For $\lambda$ nearby $\lambda_0$ we still have that the repelling fixed point of $f_\lambda$ is contained in $U$, its attracting fixed point does not lie in $\overline{U}$ and $\overline{U}\subset f_{\lambda}(U)$. In other words $(\mathcal{T},U)$ is a fast implementer for $\lambda$.
Therefore applying the algorithm of Lemma~\ref{alg: fast implementation} to $\lambda$ we obtain that the collection of values  $\{R_{G,v}(\lambda)\mid (G,v)\in \mathcal{G}_\Delta^1\}$ is dense in $\hat{\mathbb{C}}$ and hence $\lambda\in \mathcal{D}_\Delta$.
\end{proof}

For our next corollary we first need a result about the set
\[
\mathcal{E}_\Delta:=\{\lambda\in \CQ\mid Z_G(\lambda)=0\text{ for some } G\in \mathcal{G}_\Delta\}.
\]
\begin{lemma}\label{lem:exceptional set}
Let $\Delta\geq 3$ be an integer. Then the collection $\mathcal{E}_\Delta$ is contained in the set 
\[\Big\{(a+ib)^{-1}\mid a,b\in \mathbb{Z},\text{ } 0 < \sqrt{a^2+b^2}\leq \frac{\Delta^{\Delta}}{(\Delta-1)^{\Delta-1}}\Big\}.\]
\end{lemma}
\begin{proof}
Let $\lambda\in \mathcal{E}_\Delta$. Then there exists a graph $G=(V,E)$ such that $1/\lambda$ is a root of $P(z):=z^{|V|}Z_G(1/z)$.
Now $P$ is a monic polynomial and therefore $1/\lambda\in \mathbb{Z}[i]$ (since $\mathbb{Z}[i]$ is integrally closed by Gauss's lemma). 
We also know that $|1/\lambda|\leq \frac{\Delta^{\Delta}}{(\Delta-1)^{\Delta-1}}$ by Lemma~\ref{lem: RegionAroundlambda=0}.
This proves the lemma.
\end{proof}

\begin{corollary}\label{cor:density to tree}
Let $\Delta\geq 3$ be an integer.
Let $\lambda_0\in (\mathcal{D}_\Delta\cap \CQ)\setminus \mathcal{E}_\Delta$. 
Then given $P\in \CQ$ and rational $\varepsilon>0$ there exists an algorithm that generates a rooted tree $(T,v)$ such that $|R_{T,v}(\lambda_0)-P|\leq \varepsilon$ and $Z^{\text{out}}_{T,v}(\lambda_0)\neq 0$, and outputs $Z^{\text{in}}_{T,v}(\lambda_0)$ and $Z^{\text{out}}_{T,v}(\lambda_0)$ in time bounded by $poly(\size{\varepsilon,P})$. 
\end{corollary}

\begin{proof}
We first perform a brute force, but constant time, computation to obtain a fast implementer
$(\{(G_1,v_1),\ldots, (G_m,v_m),(\overline{G}_1,\overline{v_1}),\ldots,(\overline{G}_M,\overline{v}_M\},U)$ for $\lambda_0$.
Denote for $i=1,\ldots,M$, $\mu_i:=R_{G_i,v_i}$ and $\chi_i:=R_{\overline{G}_i,\overline{v_i}}$ and $g_i:= f_{\mu_i} \circ f_{\chi_i}$.


The algorithm of Lemma~\ref{alg: fast implementation} applied to $P$ now returns in time $poly(\size{\varepsilon,P})$ a sequence of ratios $\omega_1\ldots,\omega_K\in \{\lambda_0\}\cup\bigcup_{i=1}^M\{\mu_i,\chi_i\}$ that, by Lemma~\ref{lemma: paths}, correspond to the implementation of the trees $G_i$ and $\overline{G}_i$ on a path with $K=\mathcal{O}(\max{(\log(1/\epsilon),\log(|P|/\epsilon))}$ vertices.
The resulting rooted tree $(T,v)$ has maximum degree at most $\Delta$ and root degree $1$ and satisfies $|R_{T,v}(\lambda_0)-P|\leq \varepsilon$.
Denote the rooted tree corresponding to the sequence $\omega_1,\ldots\omega_i$ by $(T_i,u_i)$.
Then $(T_{i+1},u_{i+1})$ is obtained from $(T_i,u_i)$ by adding the edge $\{v_{i+1},u_i\}$ to $T_i$ and gluing a rooted tree $(H,v)\in \{K_1,(G_j,v_j),(\overline{G_j},v_j)\mid j=1,\ldots,M\}$ to $u_{i+1}$ (here $K_1$ denotes a single vertex.)
We then have
\begin{equation}\label{eq:easy recurrence}
\left(Z^\text{in}_{T_{i+1},u_{i+1}}(\lambda_0),Z^\text{out}_{T_{i+1},u_{i+1}}(\lambda_0)\right)=\left(Z^{\text{in}}_{H,v}(\lambda_0)Z^{\text{out}}_{T_i,u_i}(\lambda_0),Z^{\text{out}}_{H,v}(\lambda_0)Z_{T_i}(\lambda_0)\right).
\end{equation}
Note that~\eqref{eq:easy recurrence} describes a simple recurrence to compute $Z^{\text{in}}_{T,v}(\lambda_0)$ and $Z^{\text{out}}_{T,v}(\lambda_0)$ in time linear in the number of vertices of $T$.

Finally, we remark that $Z^{\text{out}}_{T,v}(\lambda_0)\neq 0$ since $\lambda_0\notin \mathcal{E}_\Delta$ by assumption.
\end{proof}

We can now prove the desired inclusion of the density-locus in the $\#\mathcal{P}$-locus.
\begin{theorem}
For any integer $\Delta\geq 3$ the density-locus $\overline{ \mathcal{D}_\Delta}$ is contained in the $\#\mathcal{P}$-locus $\overline{\#\mathcal{P}_\Delta}.$
\end{theorem}
\begin{proof}
We will show that for any $\lambda_0\in (\mathcal{D}_\Delta\cap \CQ)\setminus \mathcal{E}_\Delta$ the computational problem \#Hard-CoreNorm($\lambda_0,\Delta$) is \#P-hard.
Since $ \mathcal{D}^1_\Delta$ is an open set and $\mathcal{E}_\Delta$ is finite, this implies the theorem.

This in fact follows directly from the work of~\cite{Galanisetal20}. 
Let us briefly indicate why. 
In~\cite[Section 6]{Galanisetal20} the authors show that a polynomial time algorithm for \#Hard-CoreNorm($\lambda_0,\Delta$) combined with the statement of Corollary~\ref{cor:density to tree} for $\lambda_0$ yields an algorithm that on input of a graph $G$ of maximum degree at most $\Delta$ exactly computes $Z_G(1)$, the number of independent sets of $G$, in polynomial time in the number of vertices of $G$. (The algorithm is obtained by cleverly utilizing Corollary~\ref{cor:density to tree} for suitable choices of $P$ and gluing combinations of the obtained trees to $G$ and applying the assumed algorithm for \#Hard-CoreNorm($\lambda_0,\Delta$) to the resulting graph.)
Since determining $Z_G(1)$ is a known \#P-complete problem, this implies that \#Hard-CoreNorm($\lambda_0,\Delta$) is \#P-hard.
\end{proof}

We note that our result does not allow us to say anything about the complexity of \#Hard-CoreNorm($\lambda_0,\Delta$) for $\lambda_0 \in \partial (\mathcal{D}_\Delta)\cap \CQ$.
For example, for $\lambda_0\in \partial (\mathcal{D}_\Delta)\cap \mathbb{Q}_{\leq 0}$ it follows from~\cite{Galanisetal20} that the problem \#Hard-CoreNorm$(\lambda,\Delta)$ is \#P-hard. For $\lambda\in \mathbb{Q}$ such that $\lambda\geq \lambda_c(\Delta):=\tfrac{(\Delta-1)^{\Delta-1)}}{(\Delta-2)^\Delta}$ we know from~\cite{Galanisetal20} that $\lambda\in \partial( \mathcal{D}_\Delta)$, while the complexity of \#Hard-CoreNorm$(\lambda_c(\Delta),\Delta)$ is unknown. For $\lambda>\lambda_c(\Delta)$ the problem Hard-CoreNorm$(\lambda,\Delta)$ is only known to be NP-hard~\cite{Slysun}, and unlikely to be \#P-hard~cf.\cite{Galanisetal20}.

\subsection{Proofs of Lemma~\ref{lem: fast implementers exist} and Lemma~\ref{alg: fast implementation}}


The next lemma directly implies Lemma~\ref{lem: fast implementers exist}.

\begin{lemma}
\label{lem: Fast implementation requirements}
Given $z_0 \in \mathbb{C} \setminus \{-1,0\}$, a dense subset $D$ of $\mathbb{C}^*$ and a non-empty open subset $A$ of the unit disk $\mathbb{D}$ then there exists a finite set of tuples $\{(\mu_i, \chi_i)\}_{i=1}^M \subset D \times D$ and an arbitrarily small open disk $U\subseteq \mathbb{C}$ containing $z_0$ such that the maps $g_i := f_{\mu_i} \circ f_{\chi_i}$ are loxodromic Möbius transformations and
\begin{enumerate}
    \item the attracting fixed point $z_i$ of $g_i$ lies in $U$ for all $i$,
    \item $\overline{U} \subseteq \cup_{i=1}^M g_i(U)$,
    \item $g_i'(z) \in A$ for all $i$ and all $z \in \overline{U}$.
\end{enumerate}
\end{lemma}
\begin{proof}
We denote $g_{\mu, \chi} =  f_{\mu} \circ f_{\chi}$ throughout this proof. Note $g_{\mu, \chi}$ is a Möbius transformation for $\mu,\chi \neq 0$. Without loss of generality assume that $A$ is bounded away from $0$.
Take $\alpha \in A$ such that $\alpha \neq \frac{z_0}{z_0+1}$. Note that $\chi_0 = \frac{(z_0+1)^2\alpha}{z_0 - (z_0+1)\alpha}$ and $\mu_0 = \frac{z_0(z_0+\chi_0+1)}{z_0+1}$ are nonzero and well defined as $z_0 \neq -1, 0$ and $\alpha \neq \frac{z_0}{z_0+1}, 0$. Furthermore we have $g_{\mu_0,\chi_0}(z_0) = z_0$ and $g_{\mu_0,\chi_0}'(z_0) = \alpha$.

Define $F:(\mathbb{C}^*)^2 \times \mathbb{C} \rightarrow \hat{\mathbb{C}}$ as 
$F(\mu,\chi, z) =  g_{\mu,\chi}(z) - z$.
Now as $\frac{\partial F}{\partial z} (\mu_0, \chi_0, z_0) = \alpha -1 \neq 0$, the implicit function theorem gives an open neighborhood $W$ of $(\mu_0, \chi_0)$ and a holomorphic function $h:W \rightarrow \mathbb{C}$ with $h(\mu_0, \chi_0) = z_0$ and $F(\mu, \chi, h(\mu, \chi)) =0$ for all $(\mu, \chi) \in W$. 
As $h$ is a non-constant holomorphic map, it is an open map and so $h(W)$ is an open neighborhood of $z_0$.

Let $B \subseteq A$ be an open set in $\mathbb{C}$ with $\alpha \in B$ and $\overline{B} \subseteq A$. Denote $H(\mu, \chi, z) = \frac{\partial g_{\mu, \chi} }{\partial z} (z) = \frac{\mu \chi}{(1+z+\chi)^2}$, note that $H$ is continuous as a function on $\mathbb{C}^3 \setminus \{(\mu, \chi, z): \chi + z+1 = 0\}$. It follows there is an open neighborhood $C$ of $z_0$ such that we have $H(\mu_0, \xi_0, z) \in B$ for all $z \in C$.
We have $\{(\mu_0, \chi_0)\} \times \overline{C} \subseteq H^{-1}(\overline{B}) \subseteq H^{-1}(A)$. As $H^{-1}(A)$ is an open subset of $\mathbb{C}^3 \setminus \{(\mu, \chi, z): \chi + z+1 = 0\}$ containing the set $\{(\mu_0, \chi_0)\} \times \overline{C}$, by a compactness argument it follows that $H^{-1}(A)$ contains a set of the form $L \times C$, for some  open neighborhood $L$ of the point $(\mu_0, \chi_0)$. Hence the set $Y := L \cap W \cap h^{-1}(C)$ is an open neighborhood of $(\mu_0, \chi_0)$ and so $h(Y)$ is an open neighborhood of $z_0$.

Take $U \subset h(Y)$ an open disk containing $z_0$, such that $\overline{U} \subset h(Y)$. Note that we can take $U$ arbitrarily small. 
By construction, we have for all $(\mu, \chi) \in Y$ that $g_{\mu, \chi}'(z) \in A$ for all $z \in \overline{U}$. Furthermore, we have $F(\mu, \chi, h(\mu, \chi) ) = 0$, so $h(\mu, \chi)$ is the attracting fixed point of $g_{\mu, \chi}$. 
Note $D \times D$ is dense in $h^{-1}(U)$, hence the fixed points of $g_{\mu, \chi}$ for $(\mu, \chi) \in  h^{-1}(U) \cap (D\times D )$ lie dense in $U$. There is a uniform lower bound on the diameters of the disks $g_{\mu, \chi}(U)$ for $(\mu, \chi) \in h^{-1}(U)$, because $g'_{\mu, \chi}(z) \in A$ for all $z \in U$ and $A$ is bounded away from $0$.
Therefore 
\[
    \left\{g_{\mu, \chi}(U): (\mu, \chi) \in h^{-1} (U) \cap (D \times D) \right\}
\] 
is an open cover of $\overline{U}$.
As $\overline{U}$ is compact, there is a finite set of tuples $\{(\mu_i, \chi_i)\}_{i=1}^M \subseteq h^{-1} (U) \cap (D \times D)$ such that $\overline{U} \subseteq \cup_{i=1}^M g_{\mu_i,\chi_i}(U)$.
We thus found the desired set of tuples in $D\times D$ and the open disk $U$ containing $z_0$.
\end{proof}

We next focus on proving Lemma~\ref{alg: fast implementation}. To this end let $\lambda_0 \in \mathbb{C} \setminus \mathbb{R}$ and let $(\{(G_1,v_1),\ldots, (G_m,v_m),(\overline{G}_1,\overline{v_1}),\ldots,(\overline{G}_M,\overline{v}_M)\},U)$ be a fast implementer for $\lambda_0$. We fix these throughout this section.
We denote the repelling fixed point of $f_{\lambda_0}$ by $z_0$ and we denote for $i=1,\ldots,M$, $\mu_i:=R_{G_i,v_i}$,  $\chi_i:=R_{\overline{G}_i,\overline{v_i}}$ and $g_i:= f_{\mu_i} \circ f_{\chi_i}$. 
We distinguish between the case that $P$ is close to the attracting fixed point of
$f_{\lambda_0}$ and the case that it is not. In the first case the algorithm
is much simpler. 

Let $a$ be the attracting fixed point of $f_{\lambda_0}$. Because $f_{\lambda_0}(\infty) = 0$ we observe that $\infty$ is not a fixed point and thus $a \in \mathbb{C}$. Suppose that 
$|P-a| \leq \epsilon/2$. Choose $\delta >0$ for which there is a constant $\eta < 1$ 
such that $|f_{\lambda_0}'(z)| < \eta$ for all $z \in B(a,\delta)$. 
The point $0$ is not a fixed point of
$f_{\lambda_0}$ because $f_{\lambda_0}(0) = \lambda_0 \neq 0$ and thus 
$f_{\lambda_0}^n(0)$ converges to $a$ as $n \to \infty$. It follows
that there is a constant $N_0$ such that $f_{\lambda_0}^{N_0}(0) \in B(a,\delta)$. Note
that the value of $N_0$ does not depend on the input parameters.
Now let $N_{\epsilon} = \max{\{\ceil{\log_{\eta}(\frac{\epsilon}{2\delta})},0\}}+1$. Then 
for any $w \in B(a,\delta)$ we have 
\[
    |f_{\lambda_0}^{N_\epsilon}(w) - a| < \eta^{N_\epsilon} |w-a| < \epsilon/2
\]
and thus for $K = N_0 + N_\epsilon$ we have 
\[
    |f_{\lambda_0}^K(0) - P| \leq |f_{\lambda_0}^{N_\epsilon}(f_{\lambda_0}^N(0)) - a|
    + |a - P| < \epsilon/2 + \epsilon/2 < \epsilon.
\]
Because $K = \mathcal{O}(\log(1/\epsilon))$ this describes the algorithm when
$|P-a| \leq \epsilon/2$. 

The case that $|P-a| > \epsilon/2$ is more involved and we will describe the algorithm
as a sequence of simpler subroutines. Just as in Lemma~\ref{lem: Fast implementation requirements} 
let $z_i$ denote the attracting fixed point of $g_i$. We will show first show 
that, given 
a parameter $Q$ that is at most distance $\epsilon$ away from some $z_i$, we only 
have to apply $g_i$ to the starting value $0$ an $\mathcal{O}(\log(1/\epsilon))$ number of
times to get $\epsilon$ close to $Q$. Morally, this should be true because after a 
fixed number of steps the orbit of $0$ converges exponentially quickly to $z_i$ 
and because 
$z_i$ is close to $Q$ the orbit should also get close to $Q$. The only 
way that this reasoning could be incorrect is if $z_i$ and $Q$ 
are almost $\epsilon$ apart and
the orbit of $0$ converges to $z_i$ from the wrong direction. An example of this
is given by the red orbit in Figure~\ref{fig:Rotating Orbit}. This is the 
reason that we required $g_i'(z_i)$ to have an argument close to $2\pi / 3$
in which case the above reasoning is correct as the green orbit in 
Figure~\ref{fig:Rotating Orbit} demonstrates. In the following proof most 
time is spent on making this precise.


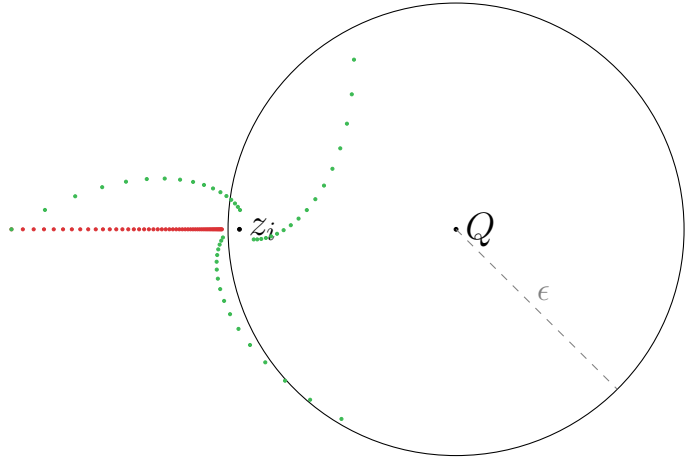
\begin{figure}[h!]
\centering
\begin{tikzpicture}
    \draw[black] (2.85, 0) circle (3);
    
    \filldraw[black] (2.85, 0) circle (0.025) node[anchor=west] {\LARGE$Q$};
    
    \draw[dashed,gray] (2.85,0) -- (4.97132, -2.12132);
    \node[gray] at (4, -0.85){\Large$\epsilon$};
    
    \filldraw[black] (0, 0) circle (0.025) node[anchor=west] {\LARGE$z_i$};
    
    
    \filldraw[gR] (-3, 0) circle (0.02);
    \filldraw[gR] (-2.85, 0) circle (0.02);
    \filldraw[gR] (-2.7075, 0) circle (0.02);
    \filldraw[gR] (-2.57213, 0) circle (0.02);
    \filldraw[gR] (-2.44352, 0) circle (0.02);
    \filldraw[gR] (-2.32134, 0) circle (0.02);
    \filldraw[gR] (-2.20528, 0) circle (0.02);
    \filldraw[gR] (-2.09501, 0) circle (0.02);
    \filldraw[gR] (-1.99026, 0) circle (0.02);
    \filldraw[gR] (-1.89075, 0) circle (0.02);
    \filldraw[gR] (-1.79621, 0) circle (0.02);
    \filldraw[gR] (-1.7064, 0) circle (0.02);
    \filldraw[gR] (-1.62108, 0) circle (0.02);
    \filldraw[gR] (-1.54003, 0) circle (0.02);
    \filldraw[gR] (-1.46302, 0) circle (0.02);
    \filldraw[gR] (-1.38987, 0) circle (0.02);
    \filldraw[gR] (-1.32038, 0) circle (0.02);
    \filldraw[gR] (-1.25436, 0) circle (0.02);
    \filldraw[gR] (-1.19164, 0) circle (0.02);
    \filldraw[gR] (-1.13206, 0) circle (0.02);
    \filldraw[gR] (-1.07546, 0) circle (0.02);
    \filldraw[gR] (-1.02168, 0) circle (0.02);
    \filldraw[gR] (-0.970601, 0) circle (0.02);
    \filldraw[gR] (-0.922071, 0) circle (0.02);
    \filldraw[gR] (-0.875967, 0) circle (0.02);
    \filldraw[gR] (-0.832169, 0) circle (0.02);
    \filldraw[gR] (-0.79056, 0) circle (0.02);
    \filldraw[gR] (-0.751032, 0) circle (0.02);
    \filldraw[gR] (-0.713481, 0) circle (0.02);
    \filldraw[gR] (-0.677807, 0) circle (0.02);
    \filldraw[gR] (-0.643916, 0) circle (0.02);
    \filldraw[gR] (-0.61172, 0) circle (0.02);
    \filldraw[gR] (-0.581134, 0) circle (0.02);
    \filldraw[gR] (-0.552078, 0) circle (0.02);
    \filldraw[gR] (-0.524474, 0) circle (0.02);
    \filldraw[gR] (-0.49825, 0) circle (0.02);
    \filldraw[gR] (-0.473338, 0) circle (0.02);
    \filldraw[gR] (-0.449671, 0) circle (0.02);
    \filldraw[gR] (-0.427187, 0) circle (0.02);
    \filldraw[gR] (-0.405828, 0) circle (0.02);
    \filldraw[gR] (-0.385536, 0) circle (0.02);
    \filldraw[gR] (-0.36626, 0) circle (0.02);
    \filldraw[gR] (-0.347947, 0) circle (0.02);
    \filldraw[gR] (-0.330549, 0) circle (0.02);
    \filldraw[gR] (-0.314022, 0) circle (0.02);
    \filldraw[gR] (-0.298321, 0) circle (0.02);
    \filldraw[gR] (-0.283405, 0) circle (0.02);
    \filldraw[gR] (-0.269234, 0) circle (0.02);
    \filldraw[gR] (-0.255773, 0) circle (0.02);
    \filldraw[gR] (-0.242984, 0) circle (0.02);
    \filldraw[gR] (-0.230835, 0) circle (0.02);
        
    
    \filldraw[gG] (-3., 0) circle (0.02);
    \filldraw[gG] (1.34195, -2.51429) circle (0.02);
    \filldraw[gG] (1.50694, 2.24937) circle (0.02);
    \filldraw[gG] (-2.55928, 0.256784) circle (0.02);
    \filldraw[gG] (0.929597, -2.25979) circle (0.02);
    \filldraw[gG] (1.4781, 1.78993) circle (0.02);
    \filldraw[gG] (-2.16132, 0.438121) circle (0.02);
    \filldraw[gG] (0.599606, -2.00737) circle (0.02);
    \filldraw[gG] (1.41416, 1.40046) circle (0.02);
    \filldraw[gG] (-1.8063, 0.558754) circle (0.02);
    \filldraw[gG] (0.339699, -1.7638) circle (0.02);
    \filldraw[gG] (1.32628, 1.07368) circle (0.02);
    \filldraw[gG] (-1.49311, 0.631278) circle (0.02);
    \filldraw[gG] (0.138822, -1.53376) circle (0.02);
    \filldraw[gG] (1.22334, 0.802422) circle (0.02);
    \filldraw[gG] (-1.21973, 0.666341) circle (0.02);
    \filldraw[gG] (-0.0128531, -1.32032) circle (0.02);
    \filldraw[gG] (1.1123, 0.579828) circle (0.02);
    \filldraw[gG] (-0.983505, 0.672852) circle (0.02);
    \filldraw[gG] (-0.123977, -1.12525) circle (0.02);
    \filldraw[gG] (0.998528, 0.399439) circle (0.02);
    \filldraw[gG] (-0.781428, 0.658187) circle (0.02);
    \filldraw[gG] (-0.202079, -0.949331) circle (0.02);
    \filldraw[gG] (0.886026, 0.25529) circle (0.02);
    \filldraw[gG] (-0.610292, 0.62838) circle (0.02);
    \filldraw[gG] (-0.25365, -0.79257) circle (0.02);
    \filldraw[gG] (0.777713, 0.141947) circle (0.02);
    \filldraw[gG] (-0.466849, 0.588304) circle (0.02);
    \filldraw[gG] (-0.284226, -0.654423) circle (0.02);
    \filldraw[gG] (0.67561, 0.0545254) circle (0.02);
    \filldraw[gG] (-0.347909, 0.541837) circle (0.02);
    \filldraw[gG] (-0.298486, -0.533955) circle (0.02);
    \filldraw[gG] (0.581024, -0.0113135) circle (0.02);
    \filldraw[gG] (-0.25042, 0.492016) circle (0.02);
    \filldraw[gG] (-0.30034, -0.429964) circle (0.02);
    \filldraw[gG] (0.494699, -0.059384) circle (0.02);
    \filldraw[gG] (-0.171518, 0.441169) circle (0.02);
    \filldraw[gG] (-0.29302, -0.341091) circle (0.02);
    \filldraw[gG] (0.41694, -0.0930036) circle (0.02);
    \filldraw[gG] (-0.108558, 0.391039) circle (0.02);
    \filldraw[gG] (-0.279169, -0.265901) circle (0.02);
    \filldraw[gG] (0.347728, -0.115028) circle (0.02);
    \filldraw[gG] (-0.0591395, 0.342884) circle (0.02);
    \filldraw[gG] (-0.260916, -0.202943) circle (0.02);
    \filldraw[gG] (0.286798, -0.127893) circle (0.02);
    \filldraw[gG] (-0.0211024, 0.297573) circle (0.02);
    \filldraw[gG] (-0.239956, -0.150796) circle (0.02);
    \filldraw[gG] (0.233718, -0.133653) circle (0.02);
    \filldraw[gG] (0.00746844, 0.255664) circle (0.02);
    \filldraw[gG] (-0.217612, -0.108103) circle (0.02);
    \filldraw[gG] (0.187943, -0.134023) circle (0.02);

\end{tikzpicture}
\caption{An example of two orbits with the same initial value converging to 
$z_i$ under iteration of two different maps. For the red orbit the derivative
at $z_i$ is real. For the green orbit the derivative at $z_i$ has 
the same magnitude, while its argument is a little less than $2\pi /3$.}
\label{fig:Rotating Orbit}
\end{figure}

\begin{lemma}
\label{lem: Quickly to $z_i$}
There exists an algorithm that,
given $\epsilon > 0$, $Q \in \mathbb{C}$ and $i \in \{1, \dots, M\}$ 
such that $|Q-z_i|<\epsilon$, yields an integer $K$ such that $|g_i^K(0)- Q| < \epsilon$, where $K = \mathcal{O}(\log(1/\epsilon))$. If $\lambda_0 \in \CQ$ and the 
input parameters $Q,\epsilon$ lie in $\CQ$ then the algorithm runs in
$poly(\size{Q,\epsilon})$ time.

\end{lemma}

\begin{proof}
Let $\delta$ be such that $B(z_i,\delta) \subseteq U$ and let $\epsilon' = \min\{\epsilon/2,\delta\}$. Note that $g_i(0) = \frac{\mu_i}{1+\chi_i} \neq 0$ 
and thus $0$ is not a fixed point of $g_i$. Because $z_i$ is the attracting fixed 
point of $g_i$ we can find (in a similar way as described above) 
a positive integer $\tilde{K}$ that 
is $\mathcal{O}(\log(1/\epsilon')) = \mathcal{O}(\log(1/\epsilon))$ such 
that $|g_i^{\tilde{K}}(0) - z_i| < \epsilon'$. If $|Q - z_i| \leq \epsilon/2$ 
we are done because then
\[
    |g_i^{\tilde{K}}(0)- Q| \leq |g_i^{\tilde{K}}(0)- z_i| + |Q - z_i| < \epsilon' + \epsilon/2 \leq \epsilon.
\]
So from now on we assume that $|Q-z_i| > \epsilon/2$. 
Define the following sector $S$ of $B(z_i,\epsilon')$
\[
    S = \{z_i + \xi\cdot\left(\frac{Q-z_i}{|Q-z_i|}\right): |\xi|<\epsilon', -\pi/3 \leq \arg(\xi) \leq \pi/3\}.
\]
We claim that $S \subseteq B(Q,\epsilon)$. To show this note that
\[
    \left|Q - \left[z_i + \xi\cdot\left(\frac{Q-z_i}{|Q-z_i|}\right)\right]\right| 
    = 
    \left|Q - z_i\right| \cdot \left|1 - \frac{\xi}{|Q-z_i|}\right|
    <
    \epsilon \cdot \left|1 - \frac{\xi}{|Q-z_i|}\right|.
\]
If $\xi$ is as in the definition of $S$ the complex number $\xi/|Q-z_i|$ has its
argument between $-\pi/3$ and $\pi/3$. Furthermore, because $|\xi|< \epsilon/2$ and $|Q-z_i|> \epsilon/2$,
its norm is bounded above by $1$. It follows that the norm of $1- \xi/|Q-z_i|$
is at most 1. Indeed, because $|1-re^{i \phi}|^2 = 1 + r^2 - 2r\cos(\phi)$, the statement $|1-re^{i \phi}|\leq 1$ is equivalent to $r=0$ or $r \leq 2\cos(\phi)$, and the latter is satisfied for all $0\leq r \leq 1$ and $-\pi/3 \leq \phi \leq \pi/3$. The claim follows.

We now claim that for $w \in B(z_i,\epsilon')$ the intersection of $\{w, g_i(w),
g_i^2(w),g_i^3(w)\}$ with $S$ is not empty. 
Note that because $\epsilon' \leq \delta$ we have that
$B(z_i,\epsilon') \subseteq U$ and thus $g'(w) \in A$ 
for every $w \in B(z_i,\epsilon')$. It follows
that applying $g_i$ to $w$ has the effect of rotating 
around $z_i$ with an angle strictly between 
$2\pi/3-0.01 $ and $2\pi/3$ and contracting towards $z_i$. 
Therefore applying $g_i$ to $w$ three
times has the effect of rotating $w$ a little less than a 
full circle around $z_i$, with steps that
are strictly less than $2\pi/3$ radians. Because the internal 
angle of the sector $S$ is $2\pi/3$ 
the orbit $w, g_i(w), g_i^2(w), g_i^3(w)$ cannot miss $S$. 

To summarize the algorithm, 
define $\epsilon'$ and determine $\tilde{K}$
such that $g_i^{\tilde{K}}(0) \in B(z_i,\epsilon')$. 
Then determine a $j\in \{0,1,2,3\}$
such that $|g_i^{\tilde{K}+j}(0)-Q| < \epsilon$. We have 
shown that there exists at least one such $j$.
The output of the algorithm is $K = \tilde{K}+j$.

\end{proof}

We shall now describe an algorithm that does the 
following. Given a disk $D$ of radius $\epsilon$ 
inside $U$, it returns an index $i$, a disk $\tilde{D}$ of radius at 
least $\epsilon$ containing $z_i$ and a sequence of indices 
$j_1, \dots, j_K$ such that $(g_{j_1} \circ \cdots \circ g_{j_{K-1}})(\tilde{D}) 
\subseteq D$. To describe the computational complexity of this algorithm, 
we need a finite way to represent disks in the complex plane. A pleasant
way for our purposes is to represent an open disk $D$ by three distinct
points $P_1, P_2, P_3$ on its boundary. This is an unambiguous way to represent
a disk because three different points on a circle uniquely determine that circle. 
If $P_1, P_2, P_3 \in \CQ$ we say that the disk $D$ is rational and that
$\size{D} = \size{P_1}+\size{P_2}+ \size{P_3}$.

Recall that a M\"obius transformation maps generalized circles (circles and straight 
lines) to
generalized circles. In what follows, we will apply M\"obius transformations
to disks in the complex plane. We shall make sure that 
the image of the disks involved is always again a disk in the complex plane and not
the complement of a disk or a half-plane as it could in general be. Therefore,
if $D$ is a disk represented by $P_1, P_2$ and $P_3$ and $g$ is one of the 
M\"obius transformations, then $g(D)$ will be a disk represented by $g(P_1), g(P_2)$
and $g(P_3)$. Note that if $D$ is rational and $g$ has rational coefficients then
$g(D)$ is a again rational. The M\"obius transformations that we will apply
come from a fixed finite set and thus there is a fixed constant $C$ for which
$\size{g(D)} \leq C \cdot \size{D}$.

Let us denote $P_j = x_{j} + i y_{j}$ for $j \in \{1,2,3\}$. The center $c_D = x+yi$ of the disk $D$ 
is known as the circumcenter of the triangle with vertices $P_1, P_2$ and $P_3$. 
The coordinates of $c_D$ can be calculated using the well known and easy to derive formulas 
\begin{align*}
    x&= \frac{(x_1^2+y_1^2)(y_2-y_3) + (x_2^2+y_2^2)(y_3-y_1) + (x_3^2+y_3^2)(y_1-y_2)}{2(x_1(y_2-y_3)+x_2(y_3-y_1)+x_3(y_1-y_2))},\\
    y&=\frac{(x_1^2+y_1^2)(x_3-x_2) + (x_2^2+y_2^2)(x_1-x_3) + (x_3^2+y_3^2)(x_2-x_1)}{2(x_1(y_2-y_3)+x_2(y_3-y_1)+x_3(y_1-y_2))}.
\end{align*}
We note that if $D$ is rational, then $c_D$ is rational and can be computed in time linear in $\size{D}$.
We can also decide whether a given point $Q\in \CQ$ lies in a given rational disk $D$ in time linear in $\size{Q}$ and $\size{D}$.

We next need a lemma concerning a geometric construction involving disks.

\begin{figure}[h!]
\centering
\begin{tikzpicture}
    \draw[gB] (2.7, 0) circle (2.1);
    \node[gB] at (4,2.4) {\LARGE$A$};
    
    \draw[gR] (0, 0) circle (3);
    \node[gR] at (-2.75,2.4) {\LARGE$B$};
    
    \filldraw[black] (2.7, 0) circle (0.025) node[anchor=west] {\Large$c_A$};
    \filldraw[black] (0, 0) circle (0.025) node[anchor=east] {\Large$c_B$};
    
    \filldraw[black] (2.2, 2.03961) circle (0.025) node[anchor=south] {\Large$S_1$};
    \filldraw[black] (2.2, -2.03961) circle (0.025) node[anchor=north] {\Large$S_2$};
    
    \draw (0, 0) -- (2.2, 2.03961) -- (2.7, 0) -- (2.2, -2.03961) -- (0,0);
    
    \draw[dashed] (0,0) -- (2.7,0);
\end{tikzpicture}
\caption{-}
\label{fig:Sectors}
\end{figure}

\begin{figure}[h!]
\centering
\begin{tikzpicture}

    \draw[fill=yellow!30] (0,0) -- (2.82843, -2.82843) arc[start angle=-45, end angle=45,radius=4] -- (0,0);
    \draw[dashed] (2.82843, -2.82843) -- (2.82843, 2.82843);
    \draw[dashed] (0,0) -- (2.82843, 0);
    \filldraw[black] (2.82843, 2.82843) circle (0.025) node[anchor=west] {\Large$P_{i+1}$};
    \filldraw[black] (2.82843, -2.82843) circle (0.025) node[anchor=west] {\Large$P_{i}$};
    \filldraw[black] (2.82843, 0) circle (0.025) node[anchor=west] {$\frac{P_{i}+P_{i+1}}{2}$};
    
    \draw[black] (2.12132,0) circle (0.707107);
    \filldraw[black] (2.12132,0) circle (0.025) node[anchor=south] {$c_D$};
    \filldraw[black] (0,0) circle (0.025) node[anchor=east] {\Large$c_A$};
    
    \draw[fill=yellow!30] (6,0) -- (9.4641, -2.) arc[start angle=-30, end angle=30,radius=4] -- (6,0);
    \filldraw[black] (6,0) circle (0.025) node[anchor=east] {\Large$c_A$};
    \filldraw[black] (9.4641, -2.) circle (0.025) node[anchor=west] {\Large$Q_{1}$};
    \filldraw[black] (9.4641, 2.) circle (0.025) node[anchor=west] {\Large$Q_{2}$};
    \draw[dashed] (6,0) -- (10,0);
    \filldraw[black] (10,0) circle (0.025) node[anchor=west] {\Large$P_i$};
    \draw[black] (9,0) circle (1);
    \filldraw[black] (9,0) circle (0.025) node[anchor=south] {$c_D$};
    
\end{tikzpicture}
\caption{-}
\label{fig:Cases}
\end{figure}

\begin{lemma}
    \label{lem: Generate Disk}
    There exists an algorithm that, given two disks $A, B$ in the complex plane
    for which the center of $A$ is 
    contained in $B$ and $B$ is not contained in $A$,
    returns a disk $D$ contained in both $A$ and $B$, such that
    the area of $D$ is at least $1/128$ times that of $A$. Furthermore, if $A$ and $B$ 
    are rational then $D$ is rational
    and both the running time of the algorithm and $\size{D}$ are bounded 
    by a fixed constant times $\size{A,B}$.
\end{lemma}

\begin{proof}
For a disk $D$ we denote its center by $c_D$ and its radius by $r_D$ and 
recall that if $D$ is rational
then $c_D$ is rational and can be computed efficiently. For two distinct points $P,Q$ 
on the boundary of $D$ we denote the closed counterclockwise arc from $P$ to $Q$ 
by $\Arc{D}{P,Q}$ 
and we denote the sector given by the convex hull of 
$\Arc{D}{P,Q}$ and $c_D$ by $\Sec{D}{P,Q}$. We note that the internal angle of 
both $\Arc{D}{P,Q}$ and $\Sec{D}{P,Q}$ is given by the arclength of $\Arc{D}{P,Q}$
divided by $r_D$.
We claim that either a sector of $A$ whose internal angle is greater
than $2\pi/3$ is contained in the closure of $B$ or a sector of $B$ whose internal angle is greater
than $2\pi/3$ is contained in the closure of $A$.

If the boundaries of $A$ and $B$ either do not intersect or intersect in one point 
then $A$ is contained in $B$ and the claim is obvious. Otherwise let $S_1,S_2$ be
the two intersection points such that $\Arc{A}{S_1,S_2}$ is contained in $B$ and thus
$\Arc{B}{S_2, S_1}$ is contained in $A$, see Figure~\ref{fig:Sectors}. Consider the 
quadrilateral $\square c_B S_1 c_A S_2$ and suppose towards
contradiction that the internal angles at both $c_A$ and $c_B$ are at most $2\pi/3$, then 
the sum of the internal angles at $S_1$ and $S_2$ is at least $2\pi/3$ and since they are
equal by symmetry the internal angle at $S_1$ is at least $\pi/3$. By then considering the 
triangle $\triangle c_B S_1 c_A$ it should follow that $|c_B - c_A| \geq |c_B - S_1| = r_B$,
which contradicts the assumption that $c_A$ is contained in $B$. We therefore find that
either the angle $\angle S_2 c_B S_1$ or $\angle S_2 c_A S_1$ is at least $2\pi/3$. If the latter 
is the case then both $\Arc{A}{S_1,S_2}$ and $c_A$ are contained in the closure of $B$ and thus
the same is true for $\Sec{A}{S_1,S_2}$. If the angle $\angle S_2 c_A S_1$ is less than $2\pi/3$, 
then $\angle S_2 c_B S_1$ is at least $2\pi/3$. It follows that $\angle c_A c_B S_1$ is the 
largest internal angle of the triangle $\triangle c_B S_1 c_A$ and thus $|c_A - c_B| < |c_A - S_1| = r_A$,
from which it follows that $c_B$ is contained in $A$. So in this case $\Sec{B}{S_2,S_1}$ is contained
in $A$.

For the algorithm we do not need to know whether a large sector of $B$ is contained in $A$ or 
vice versa. Assume for simplicity that a sector $S$ of $A$ with internal angle at least $2\pi/3$ is 
contained in $B$. Take $S$ to be as large as possible. In the case that 
$A$ is contained in $B$ we let $S=A$.
Let $P_0$ denote one of the given (rational) points on the boundary of $A$. Now for
$i = 0,1,2$ inductively define $P_{i+1}$ as $P_i$ rotated around $c_A$ with an angle of $\pi/2$. 
Calculating these points is computationally easy because $P_{i+1} = c_A + i(P_i - c_A)$. Now one 
of the following is guaranteed to be the case.
\begin{enumerate}
    \item 
    Two consecutive points $P_i$ and $P_{i+1}$ are contained in $S$.
    \item
    There is a unique index $i$ such that $P_i \in S$.
\end{enumerate}
Determining which of the two cases is true is easy since checking membership of $S$ is equivalent
to checking membership of $B$. In the first case we note that $\Sec{A}{P_i,P_{i+1}}$ is contained 
in both the closures of $A$ and $B$. Now let $R = (P_i + P_{i+1})/2$ and $c_D = (c_A+3 R)/4$ and 
let $D$ be the disk with center $c_D$ and the point $R$ on the boundary, see Figure~\ref{fig:Cases}.
It can be checked that $D$ is now contained in $\Sec{A}{P_i,P_{i+1}}$ and its area is $1/32$ that 
of $A$. 

In the second case note that the arc $\Arc{A}{Q_1,Q_2}$ containing $P_i$
such that the internal angle of both $\Arc{A}{Q_1,P_i}$ and $\Arc{A}{P_i,Q_2}$ is $\pi/6$ contained
in $S$, otherwise, since the internal angle of $S$ is at least $2\pi/3$, $S$ has to contain 
two consecutive points $P_i$ and $P_{i+1}$. Now let $D$ be the disk with center $(c_A + 3P_i)/4$
containing $P_i$ on its boundary, see Figure~\ref{fig:Cases}. It can be checked that $D$ is contained
in $\Sec{A}{Q_1,Q_2}$ and its area is $1/16$ that of $A$.

The algorithm above is only guaranteed to successfully return a disk contained in both $A$ and $B$
if a large sector of $A$ is contained in $B$. Therefore we have to run the algorithm described above
(and let it fail if neither of the two described cases is true) and run the same algorithm with 
the roles of $A$ and $B$ reversed. If both instances of the algorithm return a disk, say $D_1$ and 
$D_2$, then at least one of them is contained in both $A$ and $B$ but the other one might not be. So
in this case we have to run one final check to see which one of the two disks is indeed contained 
in both $A$ and $B$ (which is computationally easy). 
If they both are we can return either $D_1$ or $D_2$.

In conclusion we obtain a disk $D$ contained in both $A$ and $B$ that is either at 
least $1/32$ of the area of $A$ or $1/32$ of the area of $B$. Because the area of $B$ is at least 
$1/4$ that of $A$ (otherwise $r_B < r_A/2$ and $c_A \in B$ would imply $B \subseteq A$), we can 
conclude that the area of $D$ is at least $(1/4) \cdot (1/32) = 1/128$ that of $A$.

\end{proof}

\begin{lemma}
    \label{lem: Fast into D_1}
    There exists an algorithm that, given a disk $D_1 \subseteq U$ 
    with radius $\epsilon$, returns an index
    $i\in\{1,\dots, M\}$, a sequence of indices $j_1, \dots, j_K \in \{1, \dots, M\}$ and a disk 
    $D_K$ such that $z_i \in D_K$, the radius of $D_K$ is at least $\epsilon$ 
    and 
    $(g_{j_1} \circ \cdots \circ g_{j_{K-1}})(D_K) \subseteq D_1$. Furthermore,
    $K = \mathcal{O}(\log(1/\epsilon))$. If $\lambda_0 \in \CQ$ and
    $D_1$ and $\epsilon$ are both rational then $D_K$ is rational
    and the algorithm runs in $poly(\size{D_1,\epsilon})$ time. 
\end{lemma}

\begin{proof}
    For every index $i$ let $U_i = g_i(U)$. Recall that we took $U$ as a 
    rational disk. Because the derivative of $g_i$ is bounded on $U$, the image
    $U_i$ is again a disk in the complex plane. If $\lambda_0$ is rational, the 
    coefficients of $g_i$ are rational and then $U_i$ is also rational. The 
    point $z_i$ is fixed for $g_i$ and contained in $U$, therefore, also contained 
    in $U_i$. We describe a procedure to generate a sequence of disks $\{D_n\}_{n \geq 1}$, starting with the given disk $D_1$. 
    The sequence is defined in such a way such that $D_n \subseteq U$ for all $n$, 
    which is, by assumption, the case for $D_1$. 
    
    Suppose we have arrived at disk $D_n \subseteq U$. Check if there is any index $i\in\{1,\dots, M\}$ such 
    that $z_i \in D_n$, if there is stop the procedure and let $K = n$. Otherwise, 
    let $m_n$ be the center of $D_n$ and determine an index $j_n \in \{1, \dots, M\}$ such that
    $m_n \in U_{j_n}$. Such an index must exist because $m_n \in U$ and the disks
    $U_1, \dots, U_M$ cover $U$. Because the center of $D_n$ lies in $U_{j_n}$ but 
    $U_{j_n}$ is not contained in $D_n$ ($z_{j_n}$ does not lie in $D_n$) we can 
    use Lemma~\ref{lem: Generate Disk} to generate a disk $\tilde{D}_n$ that is
    contained in both $D_n$ and $U_{j_n}$ whose area is at least $1/128$ times that of $D_n$ 
    and which can be assumed to be rational if $D_n$ is. 
    Now we define $D_{n+1} = g_{j_n}^{-1}(\tilde{D}_{n})$. 
    Because $\tilde{D}_{n} \subseteq U_{j_n}$ the disk
    $D_{n+1}$ lies in $U$ and because $\tilde{D}_{n} \subseteq D_n$ the disk
    $g_{j_n}(D_{n+1})$ lies in $D_n$. Furthermore, by the properties of the fast implementer, $g_{j_n}^{-1}$ is 
    expanding the norm on $U_{j_n}$ with a factor at least $16$, the 
    area of $D_{n+1}$ is at least $16^2 \cdot (1/128) = 2$ times that of $D_n$. 
    This means that 
    the area of $D_n$ grows exponentially with $n$ and thus, because 
    the area of $U$ is fixed,
    the procedure will terminate after $K = \mathcal{O}(\log(1/\epsilon))$ steps. 
    Note that indeed 
    $z_i \in D_K$ for some $i$, the radius of $D_K$ is at least that of $D_1$ and 
    $(g_{j_1} \circ \cdots \circ j_{K-1})(D_K) \subseteq D_1$.
\end{proof}

Recall that we had defined $a$ and $z_0$ to be the attracting and repelling fixed point
of $f_{\lambda_0}$ respectively. We have already described the algorithm in 
Lemma~\ref{alg: fast implementation} when $P$ is near $a$. What follows is 
the final lemma needed to describe the algorithm when $P$ is not near $a$.

\begin{lemma}
    \label{lem: Close to P}
    There exists a fixed positive constant $c$ and an algorithm that, 
    given $P \in \mathbb{C}$ and $\epsilon >0$ such that 
    $|P-a| \geq \epsilon/2$, yields a disk $D\subseteq U$ and a positive integer $K$ with
    $f_{\lambda_0}^K(D) \subseteq B(P,\epsilon)$, such that the radius 
    of $D$ is at least $c \cdot \min{(\epsilon, \epsilon/|P|^2)}$ and $K =
    \mathcal{O}(\log(1/\epsilon))$.
    If both $\lambda_0$ and the input parameters are in 
    $\CQ$ then $D$ is also rational and both $\size{D}$ and the running time 
    of the algorithm is polynomial in $\size{P,\epsilon}$.
\end{lemma}

\begin{proof}
Let $V$ be a compact neighborhood of $a$ such that $|f_{\lambda_0}'(z)| < \xi < 1$ for some 
constant $\xi$ for all $z \in V$. 
We first claim that there is an integer $N$ such that the complement
of $f_{\lambda_0}^N(U)$ is contained in $V$. To show this let $U_n = f_{\lambda_0}^n(U)$. 
Recall that we assumed that $\overline{U} = \overline{U_0} \subseteq U_1$ and thus inductively $\overline{U_n}
\subseteq {U_{n+1}}$. Under iteration of $f_{\lambda_0}^{-1}$ every initial point
that is not $a$ converges to $z_0$ and thus eventually lands in $U$. Therefore, 
\[\bigcup_{n=1}^\infty U_n = \Chat \setminus \{a\}.\]
For $n$ large enough the point $\infty$ is contained in $U_n$ and from then on 
the sequence $(U_n)^c$ consists of nested disks, containing $a$, whose radii must 
necessarily converge to $0$,
proving that there is an $N$ such that $(U_N)^c$ is contained in $V$. Note that 
$N$ does not depend on the input parameters. Let $D_0$ be the interior of 
$(U_N)^c$, this is a rational disk whose size also does not
depend on the input, and let $D_i = f_{\lambda_0}^i(D_0)$ for $i \in \{1,2,3\}$.
Let $\delta > 0$ be a constant smaller than the minimum distance between 
points on the boundary of $D_i$ and $D_{i-1}$. From now on we will assume that 
$\epsilon < \delta$. Finally let $h = f_{\lambda_0}^{-(N+3)}$.

If $P$ lies outside $D_2$, then let $\tilde{D}$ be the disk of radius $\epsilon$
represented by $P+\epsilon, P+i\epsilon$ and $P-\epsilon$. Note that $\tilde{D}$ lies outside
$D_3$ and thus $D = h(\tilde{D}) \subset U$. Because the derivative of a M\"obius transformation 
of the form $z \mapsto (az+b)/(cz +d)$ is $z \mapsto \frac{a d-b c}{(c z+d)^2}$ there is
a constant $c_1$ such that the radius of $D$ is at least $c_1 \cdot
\min{(\epsilon, \epsilon/|P|^2)}$.
In this case $D$ and $K=N+3$ are the output of the algorithm.

If $P$ lies inside $D_2$ we determine $N_0$ such that for $P_{N_0}:= f_{\lambda}^{-N_0}(P)$
we have $P_{N_0}\in D_1$ and
$P_{N_0}\not\in D_2$. Because $|P - a| \geq \epsilon/2$ and $D_1\subset V$
we find that $N_0 = \mathcal{O}(\log(1/\epsilon))$. Let $\tilde{D}$ be the disk of radius
$\epsilon$ represented by $P_{N_0}+\epsilon, P_{N_0}+i\epsilon$ and $P_{N_0}-\epsilon$.
Note that again $\tilde{D}$ lies outside $D_3$ and thus $D=h(\tilde{D}) \subseteq U$. 
Furthermore, because $\tilde{D} \subset D_0 \subset V$ and $f_{\lambda_0}$ is attracting 
on $V$ it follows that $f_{\lambda_0}^{N_{0}}(\tilde{D}) \subseteq B(P,\epsilon)$.
Finally, if we let $c_2$ be the minimum of $|h'(z)|$ for $z \in D_0$, we find that 
the radius of $D$ is at least $c_2 \cdot \epsilon$. So in this case
the output is the disk $D$ together with $K = N_0 + N + 3$.

\end{proof}

We are now ready to complete the proof of Lemma~\ref{alg: fast implementation}.

\begin{proof}[Proof of Lemma~\ref{alg: fast implementation}]
Recall we had defined $a$ to be the attracting fixed point of $f_{\lambda_0}$
and that we already described the algorithm in the case that $|P-a|<\epsilon/2$,
therefore we assume that $|P-a|\geq \epsilon/2$. 

It follows from Lemma~\ref{lem: Close to P} that we can generate a disk 
$D_1\subseteq U$ of radius $r= \mathcal{O}(\min{\{\epsilon,\epsilon/|P|^2\}})$
and whose size is polynomial in $\size{P,\epsilon}$
together with a positive integer $K_1$ that 
is $\mathcal{O}(\log(1/\epsilon))$ such that $f_{\lambda_0}^{K_1}(D_1)$ is 
contained in $B(P,\epsilon)$. From Lemma~\ref{lem: Fast into D_1} it follows
that we can find an index $i\in \{1,\dots, M\}$, a sequence of indices
$j_1, \dots, j_{K_2}$ and a disk $D_{2}$ such that $z_i \in D_2$, the 
radius of $D_2$ is at least $r$, its size is polynomial in $\size{r,D_1}$, which is
again polynomial in $\size{\epsilon,P}$, and such that \[(g_{j_1} \circ \cdots \circ g_{j_{K_2}})(D_2) \subseteq D_1.\]
Furthermore $K_2 = \mathcal{O}(\log(1/r)) = \mathcal{O}(\max{(\log(1/\epsilon),\log(|P|/\epsilon))})$. Finally let 
$Q$ be the center of $D_2$ and note that $\size{Q}$ is polynomial in $\size{D_2}$. 
Then, because $|Q-z_i|<r$, it follows from Lemma~\ref{lem: Quickly to $z_i$}
that we can generate a $K_3$ such that $g_i^{K_3}(0) \in D_2$, where 
$K_3 = \mathcal{O}(\log(1/r)) = \mathcal{O}(\max{(\log(1/\epsilon),\log(|P|/\epsilon))})$. Concluding,
we find that 
\[
    (f_{\lambda_0}^{K_1}\circ g_{j_1} \circ \cdots \circ g_{j_{K_2}} \circ \circ g_i^{K_3})(0) \in B(P,\epsilon).
\]
Furthermore, adding the running times of the individual algorithms, we find
that the final algorithm runs in $poly(\size{P,\epsilon})$ time.
\end{proof}

\section{Activity and zeros for Cayley trees}\label{sec: Cayley trees}

For fixed $\Delta \ge 2$ notions such as the activity-locus and the zero sets can be considered for subcollections of $\mathcal{G}_\Delta$. Particularly interesting subcollections from a physical viewpoint are given by subgraphs of regular lattices. However, it is notoriously difficult to rigorously deduce the properties for such collections.

A much simpler collection of rooted graphs in $\mathcal{G}_\Delta$ is given by finite Cayley trees, and we will describe the properties of those in this section.
The trees are uniquely determined by the conditions that every leaf has fixed distance $n$ to the root vertex $v$, and every non-leaf has down-degree $d = \Delta-1$. The root vertex therefore has degree $d$, while every other non-leaf has degree $\Delta$. We denote the Cayley tree of depth $n$ by $T_n$, and its root by $v_n$.

As an immediate consequence of Lemma \ref{lem: tree recursion} we obtain
$$
R_{T_n, v_n}(\lambda) = f_{\lambda,d} (R_{T_{n-1}, v_{n-1}}(\lambda)), 
$$
where $f_{\lambda,d}(z) = \lambda/(1+z)^d$. Since the ratio of a single point is given by $\lambda = f_{\lambda,d}(0)$, it follows inductively that
$$
R_{T_n, v_n}(\lambda) = f_{\lambda,d}^{n+1}(0).
$$
In fact, since
$$
Z^{out}_{T_n, v_n}(\lambda) = \left(Z_{T_{n-1},v_{n-1}}(\lambda)\right)^d
$$
and
$$
Z^{in}_{T_n, v_n}(\lambda) = \lambda \left(Z^{out}_{T_{n-1},v_{n-1}}(\lambda)\right)^d
$$
it follows by induction on $n$ that for $\lambda \in \mathbb C \setminus \{0\}$ the polynomials $Z^{in}_{T_n,v_n}(\lambda)$ and $Z^{out}_{T_n, v_n}(\lambda)$ cannot vanish simultaneously. For $\lambda \in \mathbb C \setminus \{0\}$ it follows that $Z_{T_n}(\lambda) = 0$ if and only if $R_{T_n,v_n}(\lambda) = f_{\lambda,d}^n(0) = -1$. 

In what follows we deduce properties of the zeros of $Z_{T_n}(\lambda)$ and the activity-locus of $f_{\lambda,d}^n(0)$ from well known results in the field of holomorphic dynamical systems, occasionally adapting the proofs to our setting. We refer the reader  to the standard references \cite{Milnor2006, CarlesonGamelin}.  

Observe that $f_{\lambda,d}(-1) = \infty$ and $f_{\lambda,d}(\infty) = 0$, and $f_{\lambda,d}^\prime(-1) = f_{\lambda,d}^\prime(\infty) = 0$. Thus if $f_{\lambda_0,d}^n(0) = -1$ for some $\lambda_0$ and $n$, then $0$ is an attracting periodic cycle of period $n+2$. 
This cycle is stable under perturbations of $\lambda_0$, i.e. the attracting cycle persists and in fact varies holomorphically for nearby parameters $\lambda \sim \lambda_0$ by the implicit function theorem.

Recall that every attracting cycle attracts the orbit of a critical point. 
But $f_{\lambda,d}$ has only one critical orbit: the orbit of $-1$, $\infty$ and $0$. Thus whenever $f_{\lambda,d}$ has an attracting cycle, the orbit $f_{\lambda,d}^n(0)$ converges to the attracting cycle. In fact, the convergence is uniform in a neighborhood of the parameter $\lambda_0$, hence $\lambda_0$ cannot lie in the activity-locus. The situation is therefore fundamentally different from the setting where the whole family of graphs $\mathcal{G}_\Delta$ is considered, as there $\lambda_0$ must lie in the activity-locus. The following however does hold:

\begin{prop}\label{prop:activity}
The activity-locus of the family $\{T_n, v_n\}$ equals the collection of accumulation points of the zeros of the collection $\{Z_{T_n}\}$.
\end{prop}
\begin{proof}
If there are no zeros in a neighborhood of some $\lambda_0$, then the family $\{R_{T_n, v_n}\}$ avoids the values $0, -1$ and $\infty$, and is normal by Montel's Theorem.

Suppose on the other hand that $\lambda_0$ is an accumulation point of zeros $\lambda_1, \lambda_2, \ldots$. Let $n_1, n_2, \ldots$ be the minimal integers for which $f^{n_i}(\lambda_i) = -1$. Since for fixed $n$ the zeros of $Z_{T_n}$ are isolated, we may assume that $n_i \rightarrow \infty$ and $(n_i)$ is strictly increasing.

When for a parameter $\lambda$ the rational function $f$ has an attracting periodic cycle, the unique critical orbit $\{f^n(0)\}_{n\geq 1}$ must converge to this periodic orbit. Since attracting periodic cycles are stable, i.e. they persist under small changes of the parameter $\lambda$, such parameters lie in a passivity component, i.e., a maximal connected open subset where the family $\{\lambda \mapsto f^n(0)\}$ is normal. The passivity component agrees exactly with the connected component where the attracting periodic cycle persists, since by \cite{MSS1983} the parameter must become active when the periodic cycle becomes neutral.

Thus, $\lambda_i$ lies in a connected component of the open set where the family $\{\lambda \mapsto f^n(0)\}$ is normal, and associated to this component is the unique period $n_i+2$. Since the sequence $\{n_i\}_{i\geq 1}$ is strictly increasing, the parameters $\lambda_i$ must all lie in distinct connected components. It follows that the limit parameter $\lambda_0$ cannot lie in an open component where the family is normal, and therefore $\lambda_0$ must be an active parameter.
\end{proof}

The activity-locus for Cayley trees of down degree $d=2,3$ and $4$ is illustrated in Figure \ref{figure:cayley}. Each of these diagrams represents the spherical derivative of the function $\lambda \mapsto f_{\lambda,d}^{120}(0)$.

\begin{figure}
\begin{subfigure}[b]{0.57\textwidth}
\centering
\includegraphics[width=\textwidth]{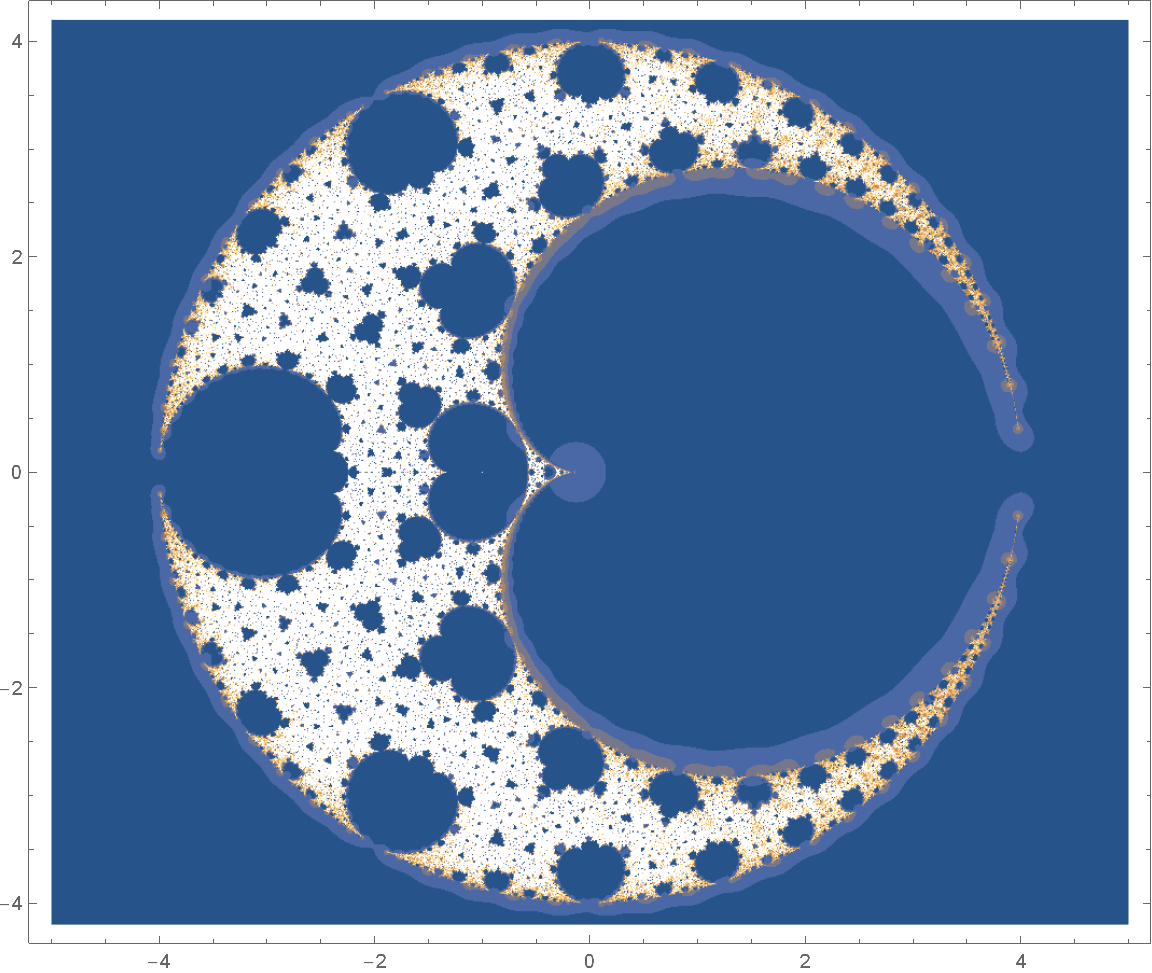}
\subcaption*{down-degree 2}
\end{subfigure}
\hfill
\begin{subfigure}[b]{0.39\textwidth}
\centering
\includegraphics[width=\textwidth]{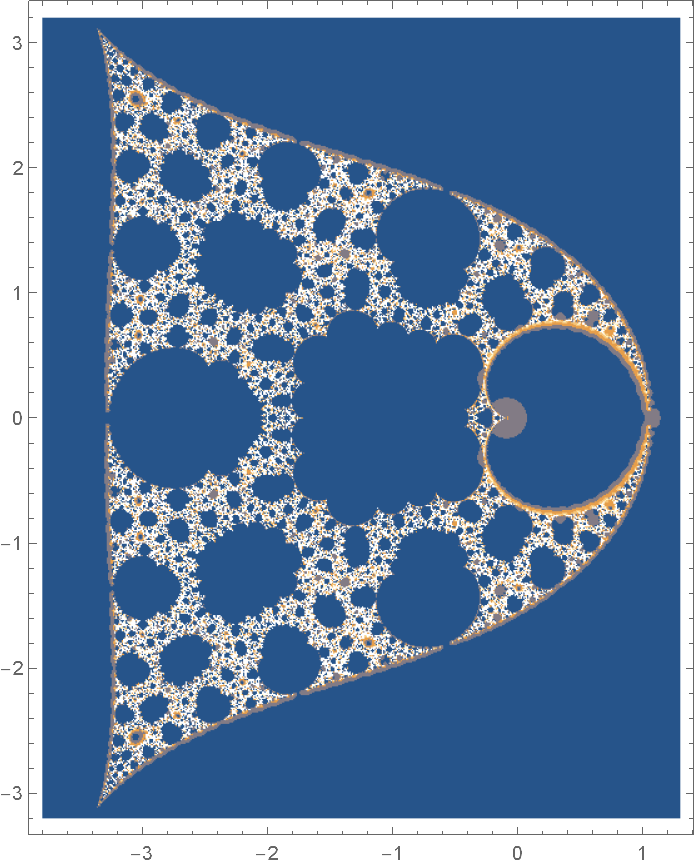} 
\subcaption*{down-degree 4}
\end{subfigure}
\hfill
\begin{subfigure}[b]{0.8\textwidth}
\centering
\includegraphics[width=\textwidth]{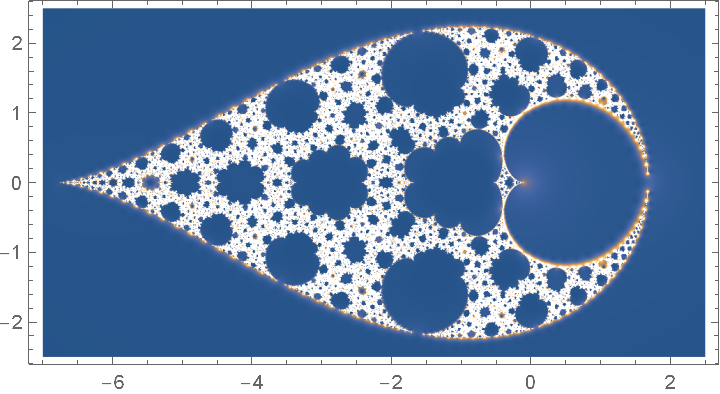}
\subcaption*{down-degree 3}
\end{subfigure}
\caption{The activity-locus of Cayley trees for down-degrees $2, 3$ and $4$. For each pixel the spherical derivative of the occupation ratio is computed for the Cayley tree of depth $120$. Pixels for which this derivative is sufficiently large are depicted in white, suggesting that the corresponding parameter $\lambda$ lies approximately on the activity locus.}
\label{figure:cayley}
\end{figure}

It follows from Proposition \ref{prop:activity} above, plus the observation that zeros do not lie in the activity-locus for the Cayley tree setting, that the Cayley tree activity-locus never has interior. On the other hand, it follows from the universality of the Mandelbrot set, a result due to McMullen~\cite{McMullen2000}, that the activity-locus must contain a quasiconformal image of the Mandelbrot set of some degree. Therefore by Shishikura's result~\cite{Shishikura1998} the Hausdorff dimension of the activity-locus is equal to $2$ for any $d \ge 2$. 

It follows from the proof of Proposition \ref{prop:activity} that the complement of the activity-locus consists of infinitely many connected components. Each $\lambda$ for which $f_{\lambda,d}$ has an attracting periodic cycle lies in such a passive component, a so-called \emph{hyperbolic} component associated to the period $k$. Whether all connected components are hyperbolic is an open question, which is conjectured to hold for quadratic polynomials. 

For any down-degree $d$ there are two special connected components that can easily be identified. The unbounded component is always a hyperbolic component of period $2$. For degree $2$ this is the complement of the closed disk of radius $4$. For down-degrees $3$ and $4$ the boundary has respectively $1$ and $2$ singular points.

For each down-degree $d=\Delta-1$ there is a single hyperbolic component of period $1$, which contains of course the parameter $\lambda = 0$ and equals the cardioid $\Lambda_\Delta$. 

Apart from these two special hyperbolic components, any hyperbolic component contains a unique zero of the partition function, i.e. a unique parameter $\lambda$ for which $f_{\lambda, d} ^n(\lambda) = -1$ for some $n \in \mathbb N$. Since $f_{\lambda,d}^2(-1) = 0$ and $f_{\lambda,d}^{-2}(0) = \{-1\}$, these are exactly the parameters $\lambda \in \mathbb C \setminus \{0\}$ for which the unique critical orbit  $\{f_{\lambda,d}^{i}(0)\}_{i \in \mathbb N}$ is periodic, i.e. for which $f_{\lambda,d}$ is super-attracting.

For the family $p_c(z) = z^2 + c$ the fact that every hyperbolic component of the Mandelbrot set contains a unique super-attracting parameter is a consequence of the Multiplier Theorem, due to Douady-Hubbard and Sullivan, see \cite{Douady1983}.

Let us recall this fundamental result in the field. Let $H$ be a hyperbolic component of the Mandelbrot set, say of period $n$. For every parameter $c\in H$ there exist an attracting periodic cycle $a_0, a_1, \ldots , a_n = a_0$. The multiplier $h(c) = (f_c^{n})^\prime(a_0)$ is independent from the choice of $a_n$, and gives a holomorphic map from $H$ to the unit disk.

\begin{thm}[Multiplier Theorem]
For every hyperbolic component $H$ the map $c \mapsto h(c)$ gives a conformal bijection from $H$ to the unit disk. 
\end{thm} 

The proof of the Multiplier Theorem can be found in \cite{CarlesonGamelin}, Theorem 2.1 on page 133, and can be applied almost directly to our setting. We present a high-level discussion to outline how the proof adapts to our setting.

Let $H$ be a hyperbolic component of period at least $3$. One easily sees that $h(\lambda)$, the multiplier of the attracting periodic cycle of $f_{\lambda,d}$ is a holomorphic and surjective map from the hyperbolic component $H$ to the unit disk $\mathbb D$, hence is a branched covering. Let $Z$ be the set of super-attracting parameters in $H$, i.e. $Z = h^{-1}(0)$. If it can be shown that $h : H \setminus Z \rightarrow \mathbb D \setminus \{0\}$ is a covering map, it follows from the Riemann-Hurwitz Theorem that $\mathrm{card}(Z) = 1$.

Thus, it needs to be shown that $h$ is locally invertible near parameters $\lambda_0 \in H \setminus Z$. Write $\eta_0 = h(\lambda_0) \in \mathbb D$, and consider values of $\eta$ near $\eta_0$. Following the proof of the Multiplier Theorem one applies quasiconformal surgery by modifying the ellipse field near the attracting periodic cycle in order to obtain attracting periodic cycles with multipliers $\eta$. Using the dynamics the ellipse field can be extended to the full basin of the attracting cycle, obtaining an invariant ellipse field that is invariant under the map $f_{\lambda_0,d}$. The ellipse field corresponds to a Beltrami coefficient, which can be extended to the entire Riemann sphere by setting it equal to $0$ outside of the basin of attraction. The Measurable Riemann Mapping Theorem gives a holomorphic family of quasiconformal maps $\varphi_\eta$, with $\varphi_{\eta_0}$ the identity. By composing with suitable M\"obius transformations we can guarantee that the points $-1, \infty$ and $0$ are fixed under all $\varphi_\eta$.

Since each ellipse field is invariant under $f_{\lambda_0,d}$, conjugating $f_{\lambda,d}$ by $\varphi_\eta$ yields a holomorphic family of self-maps of the Riemann sphere $g_{\eta,d}$ , which are necessarily rational functions of the same degree $d$. In fact, since each $\varphi_{\eta}$ fixes the points $-1, \infty$ and $0$, each rational function $f_{\eta,d}$ must send $-1$ to $\infty$ and $\infty$ to $0$, each with local degree $d$. It follows that the rational function $g_{\eta,d}$ must be of the form
$$
g_{\eta,d}(z) = \frac{\lambda(\eta)}{(1+z)^d}.
$$
It follows that $\lambda(\eta)$ gives a local inverse of the multiplier function $h$, completing this step of the proof. This step guarantees that there exists a unique zero in each hyperbolic component of period at least $3$, which equals the super-attracting \emph{center} of the hyperbolic component. The proof of the Multiplier Theorem in our setting can be concluded by analyzing the local degree near the center. We have therefore obtained the following description of the zeros of the Cayley trees:

\begin{corollary}
Every $\lambda \in \mathbb C$ for which $Z_{T_n}(\lambda) = 0$ for some $n\in \mathbb N$ is the center of a hyperbolic component of the complement of the activity locus. On the other hand: apart from the two special hyperbolic components, the unbounded component and the component containing $0$, for each center $\lambda$ of a hyperbolic component there exists an $n \in \mathbb N$ for which $Z_{T_n}(\lambda) = 0$. As a consequence zero-parameters are isolated. 
\end{corollary}


\bibliographystyle{alpha}
\bibliography{biblio}

\end{document}